\documentclass[11pt]{amsart}

\usepackage{amssymb}
\usepackage[all]{xy}
\usepackage{url}

\pdfoutput=1
\usepackage{microtype}
\usepackage{hyperref}
\usepackage{xcolor}
\hypersetup{
    colorlinks,
    linkcolor={red!50!black},
    citecolor={blue!50!black},
    urlcolor={blue!80!black}
}

\usepackage{graphicx}

\newtheorem{thm}{Theorem}[section]
\newtheorem{lem}[thm]{Lemma}
\newtheorem{cor}[thm]{Corollary}
\newtheorem{prop}[thm]{Proposition}

\theoremstyle{definition}

\newtheorem{defn}[thm]{Definition}
\newtheorem{defns}[thm]{Definitions}

\newtheorem{notation}[thm]{Notation}
\newtheorem{ex}[thm]{Example}
\newtheorem{exs}[thm]{Examples}

\theoremstyle{remark}

\newtheorem{rem}[thm]{Remark}

\numberwithin{equation}{section}

\newcommand{\thmref}[1]{Theorem~\ref{#1}}
\newcommand{\corref}[1]{Corollary~\ref{#1}}
\newcommand{\secref}[1]{\S\ref{#1}}

\newcommand{\propref}[1]{Proposition~\ref{#1}}
\newcommand{\lemref}[1]{Lemma~\ref{#1}}

\newcommand{\exref}[1]{Example~\ref{#1}}

\newcommand{\Hom}{\operatorname{Hom}}
\newcommand{\Ext}{\operatorname{Ext}}

\newcommand{\IM}{\operatorname{im}}

\newcommand{\id}{\mathrm{id}}

\newcommand{\U}{{\mathcal  U}}

\newcommand{\Z}{{\mathbb  Z}}
\newcommand{\N}{{\mathbb  N}}
\newcommand{\R}{{\mathbb  R}}
\newcommand{\Q}{{\mathbb  Q}}

\newcommand{\C}{{\mathbb  C}}

\newcommand{\Sinfty}{\Sigma^{\infty}}
\newcommand{\Oinfty}{\Omega^{\infty}}

\newcommand{\sm}{\wedge}

\newcommand{\ra}{\rightarrow}
\newcommand{\xra}{\xrightarrow}

\newcommand{\xla}{\xleftarrow}

\newcommand{\hra}{\hookrightarrow}

\begin{document}

\title[Type 2 complexes constructed from Brown--Gitler spectra]{Type 2 complexes constructed from Brown--Gitler spectra}

\author[Balderrama]{William Balderrama}
\address{Mathematical Institute \\ University of Bonn \\ Bonn, Germany}
\email{williamb@math.uni-bonn.de}

\author[Barhite]{Justin Barhite}
\address{Department of Mathematics \\ University of Colorado \\ Boulder, CO 80309}
\email{Justin.Barhite@colorado.edu}

\author[Kuhn]{Nicholas J.~Kuhn}
\address{Department of Mathematics \\ University of Virginia \\ Charlottesville, VA 22904}
\email{njk4x@virginia.edu}

\author[Larson]{Donald M.~Larson}
\address{Department of Mathematics \\ Catholic University of America \\ Washington, DC 20064}
\email{larsond@cua.edu}


\date{\today}


\begin{abstract}  In a previous paper, one of us interpreted mod 2 Dyer-Lashof operations as explicit $A$--module extensions between Brown-Gitler modules, and showed these $A$--modules can be topologically realized by finite spectra occurring as fibers of maps between 2-local dual Brown--Gitler spectra.

Starting from these constructions, in this paper, we show that infinite families of these finite spectra are of chromatic type 2, with mod 2 cohomology that is free over $A(1)$.  Applications include classifying the dual Brown--Gitler spectra after localization with respect to $K$--theory.
\end{abstract}

\maketitle

\section{Introduction} \label{introduction}

Localized at a prime $p$, finite spectra are beautifully organized by chromatic type: $X$ is of type $n$ if $K(n)^*(X) \neq 0$ while $K(n+1)^*(X) = 0$, where $K(n)$ is the $n$th Morava $K$--theory at $p$.  Closed manifolds are all of type 0 or 1, but general theory \cite{hs} tells us that finite spectra of all types abound: any finite type $n$ spectrum $X$ admits a $v_n$-self map $v: \Sigma^d X \ra X$ with cofiber of type $(n+1)$.  

Despite this, it is a bit hard to find explicit examples of spectra of higher height.  In particular, explicit $v_n$--self maps are notoriously difficult to construct, with proofs involving deep dives into either the classical Adams spectral sequence or the Adams-Novikov spectral sequence\footnote{See \cite{prasit and egger} for a recent example.}.  Some higher height examples have also been constructed as retracts of manifolds by using representation theory to construct splitting idempotents\footnote{Steve Mitchell \cite{mitchell} used the representation theory of the finite general linear groups to construct the first spectra of all heights, and a more flexible method using the symmetric groups was introduced by Jeff Smith (see \cite[Appendix C]{ravenel orange book}  or \cite{kuhn lloyd}).}, but such examples are usually quite large. 

In this paper we mine a new source of type 2 complexes.  We find infinite families of type 2 complexes at the prime 2 among the family of spectra $X(n,m)$ recently constructed by the third author \cite{kuhn dl} as the fibers of maps $f(n,m):T(n) \ra T(m)$, where $T(n)$ is the $n$-dual of the $n$th Brown-Gitler spectrum.  Thus, our type 2 complexes are built out of two finite spectra which have been much studied for rather different purposes.

\section{Main Results}

\subsection{A little bit of background}

To state our results, we need to recall a bit of background material; see \secref{HZ/2 section} and \secref{K(n) section} for more detail and references.

We work at the prime 2, and $A$ denotes the mod 2 Steenrod algebra. We let $V^{\vee}$ denote the dual of a vector space over $\Z/2$. Given $n \in \N$, $\alpha_2(n)$ denotes the number of ones in the binary expansion of $n$, and $\nu_2(n)$ denotes the greatest $i$ such that $2^i$ divides $n$.  

We recall a few facts about computing $K(1)^*(X) = KU^\ast(X;\Z/2)$. The Atiyah--Hirzebruch spectral sequence converging to $K(1)^*(X)$ has $E_2 = E_3 = H^*(X;\Z/2[v^{\pm 1}])$ with $|v| = -2$. The formula for the next differential is given by 
$$ d_3(x) = Q_1(x)v,$$
where $Q_1 = Sq^2Sq^1 + Sq^1Sq^2$ is a primitive in $A$ satisfying $Q_1^2=0$.
If $M$ is an $A$--module, let 
$$H(M;Q_1) = \ker \{Q_1: M \ra M\}/\IM \{Q_1: M \ra M\}$$
denote the $Q_1$-Margolis homology of $M$. Letting $H^*(X)$ denote mod 2 cohomology, it follows that $X$ will be $K(1)^*$--acyclic if $H(H^*(X);Q_1) = 0$.  Furthermore, J.\ Palmieri \cite[Cor.A.6]{palmieri} showed that if $M$ is a finite dimensional $A$--module, then $H(M;Q_1) = 0$ if and only if $M$ is free over $A(1)$, the 8 dimensional subalgebra of $A$ generated by $Sq^1$ and $Sq^2$.

We recall a few facts about (dual) Brown--Gitler spectra and their cohomology. Let $J(n) = H^*(T(n))$. Then $J(n)$ is an unstable $A$--module, where  we recall that an $A$--module $M$ is \textit{unstable} if $Sq^i x = 0$ for $i>n$ if $x \in M^n$.  We let $\U$ denote the category of such modules; this is a full subcategory of the abelian category of all left $A$--modules.  Note that $H^*(X) \in \U$ if $X$ is any {\em spacelike} spectrum, i.e.\ a retract of a suspension spectrum.

The modules $J(n)$ are, in fact, $\U$--injectives: $J(n)$ has a unique nonzero top degree class of degree $n$, and the natural map
\begin{equation} \label{defining property of J(n)}
 \Hom_{A}(M,J(n)) \simeq M^{n\vee}
\end{equation}
sending $f: M \ra J(n)$ to $f^n: M^n \ra J(n)^n \simeq \Z/2$ is an isomorphism for $M \in \U$. Computation shows that $J(n)$ has a unique nonzero bottom degree class in degree $\alpha_2(n)$.

The algebraic properties of $J(n)$ are reflected in two properties of $T(n)$:
\begin{itemize}
\item (Brown-Gitler property) $[T(n), X] \ra H_n(X)$ is onto if $X$ is spacelike.
\item $T(n)$ is spacelike.
\end{itemize}
We note that $T(2n+1) = \Sigma T(2n)$, and the first few even examples are $T(0) = S$, $T(2) = \Sinfty \R P^2$, and $T(4) = \Sinfty \R P^4$.

Finally we recall a bit about the results in \cite{kuhn dl}. There, a certain finite spectrum $X(n,m)$ is defined as the fiber of well-chosen map $f(n,m): T(m) \ra T(n)$, so that there is a fibration sequence of spectra
\begin{equation} \label{T(n) fibration} \Sigma^{-1}T(n) \ra X(n,m) \ra T(m) \xra{f(n,m)} T(n).
\end{equation}
By construction, the map $f(n,m)$ induces zero in mod 2 cohomology, and thus (\ref{T(n) fibration}) induces a short exact sequence of $A$--modules
\begin{equation} \label{ses equation} 0 \ra J(m) \ra Q(n,m) \ra \Sigma^{-1}J(n) \ra 0,
\end{equation}
where $Q(n,m) = H^*(X(n,m))$. To be precise, \cite{kuhn dl} works with homology instead of cohomology, and so our left $A$-module $Q(n+r,r)$ is linear dual to the right $A$-module denoted $Q(n,r)$ there; this makes no difference as linear duality defines an antiequivalence between the abelian categories of locally finite left and right $A$-modules.

Note that $Q(n,m)$ can be viewed as an element in $\Ext_A^{1,1}(J(n), J(m))$. The main result in \cite{kuhn dl} is that these induce a known action of the Dyer-Lashof algebra on the bigraded vector space $\Ext_A^{\star,\star}(M,J(*))$, for all $A$--modules $M$, by letting 
$$ Q^r: \Ext_A^{s,s}(M,J(n)) \ra \Ext_A^{s+1,s+1}(M,J(n+r))$$
be the linear map induced by Yoneda splice with $Q(n,n+r)$.  We will need a consequence of this theorem: the short exact sequence (\ref{ses equation}) is split if $m < 2n-1$.

\subsection{The main theorem and consequences}
The goal of this paper is to compute exactly when $Q(n,m)$ is $Q_1$--acyclic, and thus $X(n,m)$ is a type $2$ complex.  

\begin{thm} \label{main thm} (a) Let $m$ and $n$ be even. Then $Q(n,m)$ is $Q_1$--acyclic if and only if $\alpha_2(m) = \alpha_2(n)+1$, $\nu_2(m) = \nu_2(n)$, and $\binom{m-n-2}{n} = 1 \mod 2$. \\

\noindent (b) If $m$ and $n$ are both even, then $Q(n+1,m+1) = \Sigma Q(n,m)$, and so is $Q_1$--acyclic if and only if $Q(n,m)$ is. \\
    
\noindent (c) If $m$ and $n$ are of different parities, then $Q(n,m)$ is not $Q_1$--acyclic.

\end{thm}

\begin{rem} Recall that $\binom{b}{a} = 1 \mod 2$ exactly when every power of 2 appearing in the binary expansion of $a$ occurs in the binary expansion of $b$. This makes it clear that when $m$ and $n$ are even, $\binom{m-n-2}{n} = \binom{m/2-n/2 -1}{n/2} \mod 2$.
\end{rem}
An important special case was first proved by Brian Thomas \cite{Brian Thomas thesis} and reads as follows.

\begin{cor} \label{thomas cor} If $n$ is even and $2^k > n$, then $Q(n,n+2^k)$ is $Q_1$--acyclic.
\end{cor}
\begin{proof}  In this case the binomial coefficient is $\binom{2^k-2}{n}$, which is $1 \mod 2$.
\end{proof}
For each even $n$, the smallest $m$ such that $Q(n,m)$ is $Q_1$--acyclic is described in the next corollary.
\begin{cor} \label{minimal example}  If $n$ is even, $Q(n,2n+2^{\nu_2(n)})$ is $Q_1$--acylic.
\end{cor}
\begin{proof} Here the binomial coefficient is $\binom{n+(2^{\nu_2(n)}-2)}{n}$, which is $1 \mod 2$.
\end{proof}

\begin{ex}
Taking $n = 2$, we find that $Q(2,6)$ is $Q_1$-acyclic. The defining short exact sequence $0 \to J(6) \to Q(2,6) \to \Sigma^{-1}J(2)\to 0$ is pictured in Figure \ref{fig:Q-2-6}, with lines representing the behavior of $Sq^1$, $Sq^2$, and $Sq^4$. Classes are named according to the basis for $J(\star)$ which we recall in \thmref{J(*) calc thm}. Observe in particular that $Q(2,6)$ is isomorphic to $A(1)$ as an $A(1)$-module.
\end{ex}

\begin{figure}[tb]
\centering
\includegraphics[width=\textwidth]{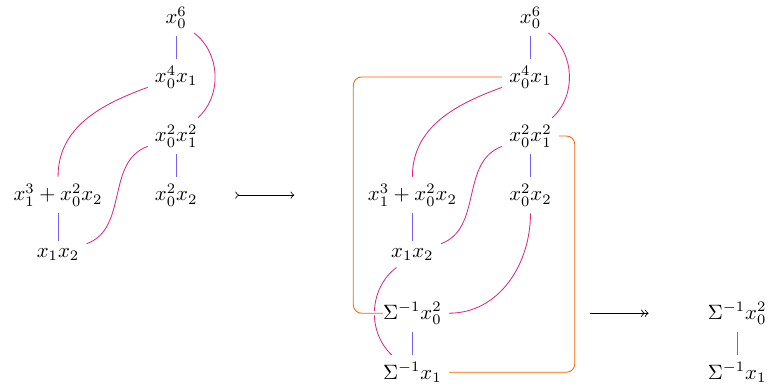}
\caption{The short exact sequence $J(6) \rightarrowtail Q(2,6) \twoheadrightarrow \Sigma^{-1} J(2)$}
\label{fig:Q-2-6}
\end{figure}

\begin{ex}
Taking $n = 4$, we find that $Q(4,12)$ is $Q_1$-acyclic. This $A$-module is pictured in Figure \ref{fig:Q-4-12}, with lines representing the behavior of $Sq^1$, $Sq^2$, $Sq^4$, and $Sq^8$. Observe in particular that $Q(4,12)$ is minimally generated as an $A$-module by the classes
\[
\Sigma^{-1}x_2 \in Q(4,12)^0,\quad \Sigma^{-1}x_0^2x_1 \in Q(4,12)^2,\quad x_0^2 x_1^3 x_2 + x_0^4x_2^2 \in Q(4,12)^6,
\]
and is freely generated as an $A(1)$-module by these classes.

A priori, the generator $\Sigma^{-1}x_0^2x_1$ is only defined modulo $J(12)\subset Q(4,12)$. However, the construction of $Q(n,m)$ gives a map $J(n) \otimes H^\ast(\tilde{P}_{-1}) \to Q(n,m)$, and for $x \in J(n)$ we may unambiguously define $\Sigma^{-1}x \in Q(n,m)$ as the image of $x\otimes t^{-1}$ under this map. We warn that as
\[
Sq^2Sq^1(\Sigma^{-1}x_1^2) = \Sigma^{-1}x_0^4 + x_1^2 x_3,
\]
the basis implicit in Figure \ref{fig:Q-4-12} does not extend any basis for $J(12) \subset Q(4,12)$.
\end{ex}

\begin{figure}[tb]
\centering
\includegraphics[width = \textwidth]{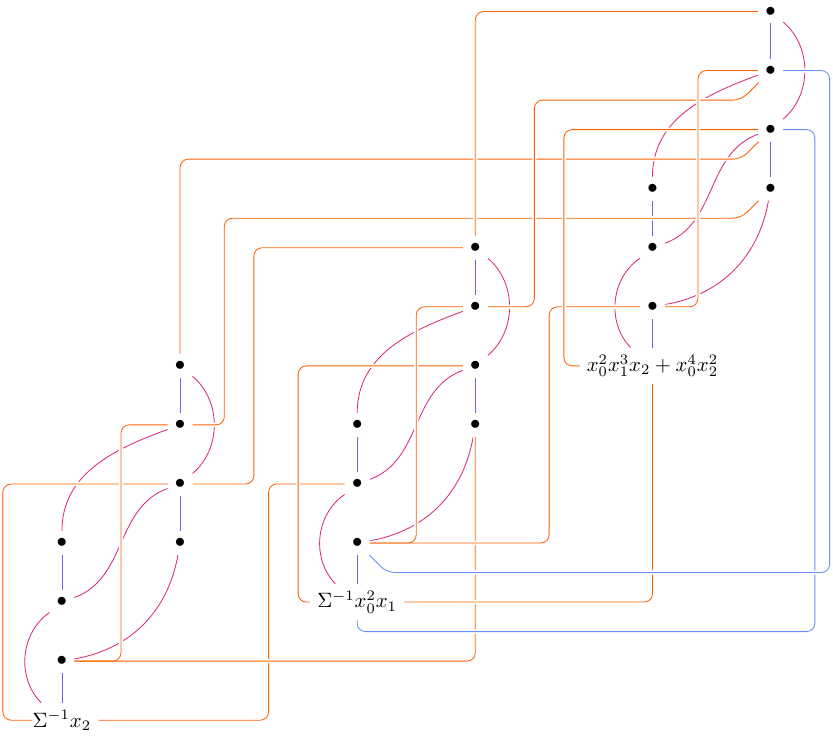}
\caption{The $A$-module $Q(4,12)$}
\label{fig:Q-4-12}
\end{figure}

\begin{ex}
The pair $(n,m) = (10,22)$ satisfies the conditions of \corref{minimal example}, and so $Q(10,22)$ is $Q_1$-acyclic. This $88$-dimensional example is the smallest instance of \thmref{main thm} that does not follow from (the substantially simpler) \corref{thomas cor}.
\end{ex}

From \thmref{main thm}, one immediately deduces the following consequences.

\begin{thm} \label{2nd main thm}  If $(n,m)$ satisfies the conditions in \thmref{main thm}(a), then
\begin{itemize}
\item[(a)] $Q(n,m)$ is a free $A(1)$--module.
\item[(b)] If $M$ is any $A$--module of the form $M = A \otimes_{A(1)} N$, with $N$ an $A(1)$--module, then Yoneda splice with $Q(n,m)$,
$$ Q(n,m) \circ: \Ext_A^{s,t}(M, J(n)) \ra \Ext_A^{s+1,t+1}(M, J(m)),$$
 will be an isomorphism for $s \geq 1$ and an epimorphism for $s=0$.
\item[(c)] $X(n,m)$ is a type 2 complex.
\item[(d)] $f(n,m): T(m) \ra T(n)$ becomes an equivalence after $L_1$--localization.
\end{itemize}
\end{thm}

Here $L_1$ is localization with respect to complex $K$--theory, localized at 2.  For statement (c), we need to know that $X(n,m)$ is not $K(2)^*$--acyclic, or, equivalently, $f(n,m)$ does not induce an isomorphism on $K(2)^*$, proved as \corref{X(n,m) not K2 acyclic cor}.  

\begin{exs} Important $A$--modules $M$ as in \thmref{2nd main thm}(b) include $H^*(ku) = A \otimes_{A(1)} (A(1) \otimes_{E(1)} \Z/2)$ and $H^*(ko) = A \otimes_{A(1)} \Z/2$.  Here $E(1)$ is the subalgebra of $A$ generated by $Sq^1$ and $Q_1$.
\end{exs}

\thmref{2nd main thm}(d) will let us classify the spectra $T(n)$ after $L_1$--localization.

\begin{thm} \label{K theory thm} (a) If $n$ is even and $i = \nu_2(n)$ then there is an $L_1$--equivalence $f:T(n) \ra T(2^i)$ of Adams filtration $\alpha_2(n)-1$.  \\

\noindent(b) There are maps $\alpha: T(2^{i-1}) \ra T(2^i)$ and $\beta: T(2^i) \ra \Sigma^2T(2)$ such that 
$$ L_1T(2^{i-1}) \xra{L_1\alpha} L_1T(2^i) \xra{L_1\beta} \Sigma^2L_1T(2)$$
is a cofibration sequence of $L_1$--local spectra. \\

\noindent(c) $K^0(T(2^i))_{(2)} \simeq \Z/2^i$, $K^1(T(2^i))_{(2)} \simeq 0$, and, on $K^0$, the cofibration sequence of part (b) induces the short exact sequence
$$ 0 \ra \Z/2 \ra \Z/2^i \ra \Z/2^{i-1} \ra 0.$$
\end{thm}

\subsection{First steps in the proof of \thmref{main thm}}

We need to say a bit more about the construction of $f(n,m): T(m) \ra T(n)$ from \cite{kuhn dl}.  Let $P_1= \Sigma^{\infty}\R P^{\infty}$, so that $H^*(P_1) = (t) \subset \Z/2[t] = H^\ast(\R P^\infty)$. Then let $\tilde P_{-1}$ denote the fiber of the Kahn-Priddy map $\tau:P_1 \ra S$, so there is a fibration sequence of spectra
\begin{equation} \label{P fibration} S^{-1} \ra \tilde P_{-1} \ra P_1 \xra{\tau} S.
\end{equation}
Smashing this sequence with $T(n)$ yields a fibration sequence of spectra
\begin{equation} \label{TP fibration} \Sigma^{-1}T(n) \ra T(n) \sm \tilde P_{-1} \ra T(n)\sm P_1 \xra{1 \sm \tau} T(n)
\end{equation}
inducing a short exact sequence of $A$--modules
\begin{equation} \label{P ses equation } 0 \ra J(n) \otimes (t) \ra J(n) \otimes H^*(\tilde P_{-1}) \ra \Sigma^{-1}J(n) \ra 0.
\end{equation}

A well-chosen $A$--module map $q(n,m)^*: J(n) \otimes (t) \ra J(m) $ is realized by a map $q(n,m): T(m) \ra T(n) \sm P_1$, and $f(n,m)$ is defined as the composite $T(m) \xra{q(n,m)} T(n) \sm P_1 \xra{1 \sm \tau} T(n)$. It follows that there is a map of fibration sequences 
\begin{equation*}
\SelectTips{cm}{}
\xymatrix{
\Sigma^{-1} T(n)  \ar[r] & T(n) \sm \widetilde P_{-1}   \ar[r] & T(n) \sm P_1   \ar[r]^{1 \sm t} & T(n)\\
\Sigma^{-1} T(n)  \ar[r] \ar@{=}[u]& X(n,r) \ar[r] \ar[u] & T(m) \ar[u]_{q(n,m)} \ar[r]^{f(n,m)} & T(n) \ar@{=}[u]}
\end{equation*}
inducing a map of $A$--module extensions
\begin{equation} \label{ses diagram}
\SelectTips{cm}{}
\xymatrix{
0  \ar[r] & J(n) \otimes (t) \ar[d]^{q(n,m)^*}  \ar[r]& J(n) \otimes H^*(\tilde P_{-1}) \ar[d]  \ar[r]&  \Sigma^{-1} J(n) \ar@{=}[d] \ar[r] & 0 \\
0  \ar[r] & J(m)   \ar[r]& Q(n,m) \ar[r]&  \Sigma^{-1} J(n)  \ar[r] & 0. }
\end{equation}
It is easy to check -- see \exref{tilde P example} -- that $H^*(\tilde P_{-1})$ is $Q_1$--acyclic, and then that diagram chasing yields the next lemma.

\begin{lem} \label{simplifying lemma} The $A$-module $Q(n,m)$ is $Q_1$--acyclic if and only if $q(n,m)^*$ is a $Q_1$--isomorphism.
\end{lem}

Thus, the proof of \thmref{main thm} amounts to analyzing exactly when 
$$ q(n,m)^*: H(J(n) \otimes (t); Q_1) \ra H(J(m); Q_1)$$
is an isomorphism.

The good news here is that the domain and range are always 2-dimensional, with classes represented by explicit elements in $J(n) \otimes (t)$ and $J(m)$.

Less good, and the cause of most of the work in this paper, is that the $A$--module map $q(n,m)^*: J(n) \otimes (t) \ra J(m)$ is defined implicitly using (\ref{defining property of J(n)}), and thus is only explicitly described in degree $m$.

\subsection{Organization of the rest of the paper}
In \secref{HZ/2 section} we recall basic properties of the dual Brown-Gitler modules; in particular, the elegant description of $ \bigoplus_{n=0}^{\infty}J(n)$, viewed as an algebra in $\U$. This allows us to precisely define our key $A$--module maps 
$$q(n,m)^*: J(n) \otimes (t) \ra J(m),$$
as well as another family of $A$--module maps 
$$p(n,l)^*: J(n) \ra J(l).$$ 
Like $q(n,m)^*$, the map $p(n,l)^*$ can also be geometrically realized by a map $p(n,l): T(l) \ra T(n)$, and we let $Y(n,l)$ denote the fiber.  

Then in \secref{K(n) section} we use the description of $\bigoplus_{n=0}^{\infty}J(n)$ to recover implicitly known calculations of $K(m)^*(T(n))$ for all $n$ and $m$, calculating the Margolis homology groups $H(J(n); Q_m)$ enroute.  Also proved here will be \lemref{simplifying lemma}.

We prove \thmref{main thm} in the three sections that follow this.

In \secref{main proof part 1}, we first recall results from \cite{kuhn dl} allowing us to conclude that if $q(n,m)^*$ is a $Q_1$--isomorphism, then necessarily $m\geq 2n$.  When this condition holds, we show that $q(n,m)^*$ is a $Q_1$--isomorphism if and only if $p(n,l)^*$ is a $Q_1$--isomorphism, where $m=l+2^k$ with $0 \leq l < 2^k$.  As $p(n,n)^*: J(n) \ra J(n)$ is the identity, this is already sufficient to deduce \corref{thomas cor}.

In  \secref{main proof part 2}, we focus on pairs $(n,l)$ when $n$ and $l$ are even, and first show that $H(J(n); Q_1) \simeq  H(J(l); Q_1)$ as graded vector spaces if and only if $\alpha_2(l) = \alpha_2(n)$ and $\nu_2(l) = \nu_2(n)$.  Assuming these, we then answer the much more delicate question of when $p(n,l)^*$ is a $Q_1$--isomorphism, leading to a technical condition involving the binary expansions of $n$ and $l$: see \thmref{p(2n,2l) iso thm}.  

Note that, if $m=l+2^k$ with $0 \leq l < 2^k$, $\alpha_2(l) = \alpha_2(n)  \Leftrightarrow \alpha_2(m) = \alpha_2(n)+1$ and $\nu_2(l) = \nu_2(n) \Leftrightarrow \nu_2(m) = \nu_2(n)$.  A combinatorial lemma, \lemref{magic lemma}, then shows that our technical condition on the pair $(n,l)$ is equivalent to the statement that $\binom{m-n-2}{n} = 1 \mod 2$, completing the proof of  \thmref{main thm}(a).

In \secref{main proof part 3},  we use similar methods to prove \thmref{main thm}(c):  $p(n,l)^*$ is never a $Q_1$--isomorphism when $n$ and $l$ are of different parities.

In the final section \secref{applications} we first prove \thmref{K theory thm} and then discuss how our work fits with the thesis of Brian Thomas \cite{Brian Thomas thesis}.  

\subsection{Acknowledgments}  This has been a collaborative research project that began at the summer 2022 research workshop funded by N.S.F. Focused Research Grant 1839968.  Our main results were presented in the Collaborative Team Conference held at the University of Virginia in December, 2023. 

Our work builds on the 2019 University of Virginia Ph.D.~ thesis work of Brian Thomas \cite{Brian Thomas thesis}, working under the supervision of the third author. In particular, knowing \corref{thomas cor} was true was useful, even though our proof of this is ultimately different.  

\section{The mod 2 cohomology of $T(n)$} \label{HZ/2 section}

\subsection{The $A$--module $H^*(\R P^{\infty})$ and related modules}

Recall that, as an unstable $A$--algebra $H^*(\R P^{\infty}) \simeq \Z/2[t]$, with $Sq^it^j = \binom{j}{i}t^{j+i}$.

For $d \in \Z$, we let $P_d$ be the Thom spectrum of $d$ copies of the canonical line bundle over $\R P^{\infty}$.  Then $P_0 = \Sinfty \R P^{\infty}_+$ with cohomology $\Z/2[t]$ as just described and  $P_1 = \Sinfty \R P^{\infty}$ with cohomology equal to the ideal $(t)$. For general $d$, $H^*(P_d)$ has basis given by $t^j$ for $j \geq d$ with the formula $Sq^it^j = \binom{j}{i}t^{j+i}$ still applying\footnote{For $j$ possibly negative, $\binom{j}{i}$ is defined by the identity $(x+1)^j = \sum_{i=0}^{\infty} \binom{j}{i}x^i$}. In particular $Sq^it^{-1} = t^{i-1}$ for all $i$.

For $d \leq 0$, let $\tilde P_d$ be the fiber of the composite $P_d \ra P_0 \xra{\pi} S$, where $\pi$ is the standard retraction.  Then $\tilde P_0 \simeq P_1$, and the cofibration sequence
$$ S^{-1} \ra P_{-1} \ra P_0 \xra{tr} S,$$
with $tr$ the $C_2$--transfer, induces the cofibration sequence
$$ S^{-1} \ra \tilde P_{-1} \ra P_1 \xra{\tau} S$$
of the introduction.  The first terms of this induce a short exact sequence of $A$--modules
\begin{equation} \label{basic P ses} 0 \ra (t) \ra H^*(\tilde P_{-1}) \ra \Sigma^{-1} \Z/2 \ra 0.
\end{equation}

\subsection{Projectives and injectives in $\U$}
A good general reference for much of the algebra here is \cite[Chapter 2]{schwartz book}. 

Let $F(m) \in \U$ be the free unstable module on a class of degree $m$. There is a natural isomorphism
\begin{equation} \label{defining property of F(m)}
 \Hom_{A}(F(m),N) \simeq N^m
\end{equation}
for all unstable modules $N$.
Particularly important for us is the explicit structure of $F(1)$: it is the sub-$A$--module of $\Z/2[t] = H^*(\R P^{\infty})$ spanned by the elements $t^{2^k}$ with $k \geq 0$.  One has $Sq^{2^k} t^{2^k} = t^{2^{k+1}}$ with all other nonidentity Steenrod operations acting as 0.

Dually, the dual Brown-Gitler module $J(n)$ is the unstable module such that there is a natural isomorphisms
\begin{equation} \label{defining property of J(n) again}
 \Hom_{A}(M,J(n)) \simeq  M^{n\vee}
\end{equation}
for all unstable modules $M$. 

Combining (\ref{defining property of F(m)}) and (\ref{defining property of J(n) again}) shows that $J(n)^m \simeq F(m)^{n\vee}$.  From this, one sees that $J(n)^n \simeq \Z/2$ and $J(n)^m = 0$ for $m \geq n$.  One also sees that $J(n)^1 \simeq \Z/2$ if $n$ is a power of 2 and is 0 otherwise.  

\begin{notation}  For $k\geq 0$, let $x_k$ be the nonzero element of $J(2^k)^1$ and let $p_k: F(1) \ra J(2^k)$ be the unique nonzero $A$--module map.  Note that $p_k(t) = x_k$.
\end{notation} 

Using (\ref{defining property of J(n) again}), the natural transformations $Sq^i: M^n \ra M^{n+i}$ induce corresponding  maps of left $A$--modules $ \cdot Sq^i: J(n+i) \ra J(n)$.  

\begin{lem} \label{Sq2^k lemma}  $\cdot Sq^{2^k}: J(2^{k+1}) \ra J(2^k)$ maps $x_{k+1}$ to $x_k$.
\end{lem}  
\begin{proof} Unwinding the definitions, one sees that the formula $\cdot Sq^{2^k}(x_{k+1}) = x_k$ is dual to $Sq^{2^k}(t^{2^k}) = t^{2^{k+1}}$.
\end{proof}

\subsection{The bigraded algebra $J(\star)^*$}

We recall the computation of the bigraded object $J(\star)^* = \bigoplus_{n = 0}^{\infty} J(n)$.   This is a graded commutative algebra in the category $\U$ with multiplication $\mu: J(m) \otimes J(n) \ra J(m+n)$ corresponding to the unique $A$--module map that is nonzero in degree $m+n$.
\begin{thm} \cite[Thm.2.4.7]{schwartz book} \label{J(*) calc thm} There is an isomorphism
$$ J(\star)^* \simeq \Z/2[x_0,x_1,\dots].$$
The $A$--module structure is determined by the instability condition, the Cartan formula, and the formulae $Sq^1(x_{k+1}) = x_{k}^2$ and $Sq^1x_0 = 0$.
\end{thm}

In particular, each $J(n)$ has a basis given by some of the monomials in the $x_k$'s.  The top degree nonzero class is $x_0^n \in J(n)^n$ and there is a unique bottom degree nonzero class in degree $\alpha_2(n)$ as follows.
\begin{notation} Let $ x(n) \in J(n)^{\alpha_2(n)}$ be the class $x(n) = x_{i_1}\cdots x_{i_d}$ where $n = 2^{i_1} + \cdots + 2^{i_d}$ with $i_1 < \cdots < i_d$ (so $d=\alpha_2(n)$).
\end{notation}

The following calculation of the map $p_k: F(1) \ra J(2^k)$ will be used by us, and is easily proved by induction on $i$.
\begin{lem} \label{pk lemma}  For $0 \leq i \leq k$, $p_k(t^{2^i}) = x_{k-i}^{2^i}$.
\end{lem} 

\subsection{The Mahowald sequences}  There are isomorphisms of $A$--modules
$$ \Sigma J(2n) \xra{\sim}J(2n+1)$$
and also short exact sequences \cite[Prop.2.2.3]{schwartz book}
\begin{equation} \label{mahowald ses}  0 \ra \Sigma J(2n-1) \ra J(2n) \xra{\cdot Sq^n} J(n) \ra 0.
\end{equation}
These are easy to describe on basis elements.  Firstly, $\Sigma J(n) \ra J(n+1)$ sends $\sigma x$ to $x_0x$. Secondly, it is not hard to check that 
$$\cdot Sq^{\star}: J(2{\star}) \ra  J(\star)$$
is a ring homomorphism, and identifies with the ring homomorphism
$$ \Z/2[x_0^2,x_1,x_2, \dots] \ra \Z/2[x_0,x_1,x_2, \dots]$$
sending $x_0^2$ to zero and $x_{k+1}$ to $x_k$.

\subsection{Two families of $A$--module maps and a family of $A$--modules}  As $J(n)$ has a standard basis, so does $J(n) \otimes (t)$ in the obvious way. We define two families of $A$--module maps using (\ref{defining property of J(n) again}).

\begin{defns} {\bf (a)} Let $p(n,l)^*: J(n) \ra J(l)$ be the unique $A$--module map which is nonzero on each standard basis element of $J(n)^l$. (In particular, $p(n,n)^*$ is the identity.)

\noindent{\bf (b)} Let $q(n,m)^*: J(n) \otimes (t) \ra J(m)$ be the unique $A$--module map which is nonzero on each standard basis element of $(J(n) \otimes (t))^m$. 
\end{defns}

Tensoring (\ref{basic P ses}) with $J(n)$ gives a short exact sequence of $A$--modules
$$ 0 \ra J(n) \otimes (t) \ra J(n) \otimes H^*(\tilde P_{-1}) \ra \Sigma^{-1}J(n) \ra 0.$$

\begin{defn}\label{qnm definition}  Define the $A$--module $Q(n,m)$ to be the pushout of the diagram $J(m) \xla{q(n,m)^*} J(n) \otimes (t) \ra J(n) \otimes H^*(\tilde P_{-1})$.
\end{defn}
Thus, there is a commutative diagram of short exact sequences
\begin{equation} \label{ses diagram again}
\SelectTips{cm}{}
\xymatrix{
0  \ar[r] & J(n) \otimes (t) \ar[d]^{q(n,m)^*}  \ar[r]& J(n) \otimes H^*(\tilde P_{-1}) \ar[d]  \ar[r]&  \Sigma^{-1} J(n) \ar@{=}[d] \ar[r] & 0 \\
0  \ar[r] & J(m)   \ar[r]& Q(n,m) \ar[r]&  \Sigma^{-1} J(n)  \ar[r] & 0. }
\end{equation}

\subsection{Geometric realization}
Brown and Gitler \cite{bg} showed that there are remarkable finite spectra $T(n)$\footnote{The notation $T(n)$ was used by them, and in subsequent papers by Lannes, Goerss, the third author, and others, and should not be confused with telescopic $T(n)$, used by many, including some of us.} (the $n$-duals of what were later termed Brown-Gitler spectra) satisfying the following properties:
\begin{enumerate}
\item $H^*(T(n)) \simeq J(n)$.
\item The natural map $[T(n),X] \ra H_n(X)$, sending $f$ to $f_*(\iota_n)$, is onto whenever $X$ is spacelike (i.e.\ a retract of a suspension spectrum).

\end{enumerate}
Later Goerss \cite{goerss T(n) is spacelike} and Lannes \cite{lannes T(n) are spacelike} proved a third property:
\begin{enumerate}
\item[(3)] $T(n)$ is spacelike.
\end{enumerate}
A nice proof of all of this is given in \cite{glm}, and the equivalence of the second and third properties, assuming the first, is shown in \cite{hunter kuhn}.

It follows that there exist maps $T(i+j) \ra T(i) \sm T(j)$ inducing $\mu: J(i) \otimes J(j) \ra J(i+j)$ in mod 2 cohomology.  

Similarly, there exist maps $p(n,l): T(l) \ra T(n)$ and $q(n,m): T(m) \ra T(n) \sm P_1$
inducing $p(n,l)^*$ and $q(n,m)^*$. We then define $Y(n,l)$ to be the fiber of $p(n,l)$ and $X(n,m)$ to be the fiber of $f(n,m)$, where  $f(n,m): T(m) \ra T(n)$ is the composite $ (1 \sm \tau) \circ q(n,m)$.  

From these constructions it follows that $Q(n,m) \simeq H^*(X(n,m))$. Indeed, applying cohomology to the left part of the diagram of fibration sequences
\begin{equation*}
\SelectTips{cm}{}
\xymatrix{
\Sigma^{-1} T(n)  \ar[r] & T(n) \sm \widetilde P_{-1}   \ar[r] & T(n) \sm P_1   \ar[r]^{1 \sm \tau} & T(n)\\
\Sigma^{-1} T(n)  \ar[r] \ar@{=}[u]& X(n,r) \ar[r] \ar[u] & T(m) \ar[u]_{q(n,m)} \ar[r]^{f(n,m)} & T(n) \ar@{=}[u]}
\end{equation*}
yields diagram (\ref{ses diagram again}).

\begin{rem} Note that $p(n,l)$ and $q(n,m)$ are only well defined up to maps of positive Adams filtration, and thus $f(n,m)$ can be varied by certain maps of Adams filtration greater than 1.  We do not know how this affects the homotopy types of the spectra $Y(n,l)$ and $X(n,m)$.
\end{rem}

Finally, we note that the Mahowald short exact sequence can be realized geometrically by a cofibration sequence: there exists a cofibration sequence
$$ T(n) \ra T(2n) \ra \Sigma T(2n-1)$$
inducing (\ref{mahowald ses}) in cohomology.  The Brown-Gitler property makes it clear that two maps exist that realize the algebraic ones, but it is not formal that they can be chosen so that the composite is null, showing that one has a cofibration sequence.  The proof that this {\em can}  be done was given in \cite{cmm78}.

\section{The Morava $K$--theory of $T(n)$} \label{K(n) section}

\subsection{The Morava $K$--theory A.H.S.S.}

Let $m \geq 1$. Recall that the $m$th Morava $K$--theory (at the prime 2) is a complex oriented ring theory having coefficient ring $K(m)^* = \Z/2[v_m^{\pm 1}]$, with $|v_m| = 2-2^{m+1}$. (This is cohomological grading.)  This is a graded field, and it follows that $K(n)^*(X)$ is a graded $K(m)^*$--vector space and, for finite $X$ and all $Y$, the natural map  $K(m)^*(X) \otimes_{K(m)^*} K^*(Y) \ra K(m)^*(X \sm Y)$ is an isomorphism. 

The Atiyah--Hirzeburch spectral sequence $\{E_r^*(X)\} \Rightarrow K(m)^*(X)$ has
$E_2^{*}(X) = H^*(X;\Z/2[v_m^{\pm 1}])$.  By sparseness, the first possible nonzero differential
is $d_{2^{m+1}-1}$, and, indeed, it is not hard to deduce\footnote{   
The proof of the odd prime version of this formula in \cite{Yagita} works without change at the prime 2, and we sketch the idea. As the differential is natural, it must have the form $ d_{2^{m+1}-2}(x) = v_m ax$ for some $a \in A^{2^{m+1}-1}$.  Furthermore $a$ must be primitive, as the differential must be a derivation when $X$ is a space, and must be nonzero to be consistent with $K(m)^*(\R P^{\infty}) \simeq K(m)^*[y]/(y^{2^m})$, calculated using a Gysin sequence. $Q_m$ is the only nonzero primitive in degree $2^{m+1}-1$.} the precise formula:
$$ d_{2^{m+1}-2}(x) = v_m Q_mx.$$
Here, $Q_m \in A^{2^{m+1}-1}$ for $m \geq 0$ are primitives
recursively defined by $Q_0 = Sq^1$ and $Q_m = [Q_{m-1}, Sq^{2^m}]$. These satisfy $Q_m^2=0$.

It follows that there is a natural isomorphism $E_{2^{m+1}}^*(X) \simeq K(m)^* \otimes H(H^*(X);Q_m)$ where the $Q_m$--homology of an $A$--module $M$ is defined to be 
$$ H(M;Q_m) = \frac{\ker Q_m}{\IM Q_m}.$$

Note that a short exact sequence of $A$--modules, $0 \ra L \ra M \ra N \ra 0$ will induce a long exact sequence
$$ \cdots \ra H(L;Q_m) \ra H(M;Q_m) \ra H(N;Q_m) \xra{d} H(L;Q_m) \ra \cdots$$
where $d$ raises degree by $|Q_m| = 2^{m+1}-1$.

\begin{ex} \label{tilde P example}  It is easy to compute that, in $H^*(\tilde P_{-1})$, we have $Q_1 t^j = t^{j+3}$ for all odd $j$, and thus $H^*(\tilde P_{-1})$ is $Q_1$--acyclic\footnote{This illustrates that the hypothesis of finite dimensionality is needed in \cite[Cor.A.6]{palmieri}, as $H^*(\tilde P_{-1})$ is {\em not} free over $A(1)$.}.  Since $Q_1$ is a derivation, $Q_1$--homology satisfies a K\"unneth theorem, and so any module of the form $M \otimes H^*(\tilde P_{-1})$ will also be $Q_1$--acyclic. This, together with the 5-lemma applied to the $Q_1$--homology long exact sequences coming from the diagram 
$$
\SelectTips{cm}{}
\xymatrix{
0  \ar[r] & J(n) \otimes (t) \ar[d]^{q(n,m)^*}  \ar[r]& J(n) \otimes H^*(\tilde P_{-1}) \ar[d]  \ar[r]&  \Sigma^{-1} J(n) \ar@{=}[d] \ar[r] & 0 \\
0  \ar[r] & J(m)   \ar[r]& Q(n,m) \ar[r]&  \Sigma^{-1} J(n)  \ar[r] & 0, }
$$
proves \lemref{simplifying lemma}: $Q(n,m)$ is $Q_1$--acyclic if and only if $q(n,m)^*$ is a $Q_1$--isomorphism.
\end{ex}  

\begin{rem}  When $m=0$, one can replace the Atiyah--Hirzebruch spectral sequence in most of the above with the Bockstein spectral sequence, and these can be identified with localized Adams spectral sequences for connective Morava $K$--theories $k(n)^*$, where $K(0) = H\Q$ and $k(0)= H\Z$. (See \cite[\S 3.1]{kuhn lloyd AGT} for more detail.)
\end{rem}

\subsection{The $Q_m$--homology groups of $J(n)$}

In this subsection, we calculate $H(J(\star)^*;Q_m)$.

Recall that $J(\star)^* = \Z/2[x_0,x_1, \dots]$, with $x_k \in J(2^k)^1$ and with $A$--module structure determined by the Cartan formula, the instability condition, $Sq^1(x_0) = 0$, and  $Sq^1(x_{k}) = x_{k-1}^2$ for $k\geq 1$.  The action of the derivation $Q_m$ on $J(\star)^*$ is determined by the following lemma, which easily proved by induction on $m$.

\begin{lem}  
$Q_mx_k =
\begin{cases}
 0 & \text{for } k\leq m, \\ x_{k-m-1}^{2^{m+1}} & \text{for } k > m.
\end{cases}
$
\end{lem}

\begin{prop} \label{Qm homology prop} $$H(J(\star)^*;Q_m) = \Z/2[x_0,\dots,x_m, x_{m+1}^2, x_{m+2}^2, \dots]/(x_k^{2^{m+1}} \ | \ k \geq 0).$$
\end{prop}
\begin{proof}  It is convenient to let $R = \Z/2[x_0, \dots, x_m, x_{m+1}^2, x_{m+2}^2, \dots]$.  Then $(J(\star)^*;Q_m)$ is a chain complex of $R$ modules that decomposes as an infinite tensor product of chain complexes of $R$--modules
$$ J(\star)^* \simeq C(0) \otimes_R C(1) \otimes_R C(2) \otimes \cdots$$
where $C(k) = R \oplus Rx_{m+1+k}$ with differential $d(x_{m+k+1}) = x_k^{2^{m+1}}$.  

Using that $x_0^{2^{m+1}}, x_1^{2^{m+1}}, x_2^{2^{m+1}}, \dots$ is a regular sequence in $R$, we see first that 
$H(C(k)) = R/(x_k^{2^{m+1}})$ and then that a K\"unneth isomorphism holds:
$$ H(J(\star)^*;Q_m) \simeq H(C(0);Q_m) \otimes_R H(C(1);Q_m) \otimes_R H(C(2);Q_m) \otimes \cdots.$$
It follows that $H(J(\star)^*;Q_m) = R/(x_k^{2^{m+1}} \ | \ k \geq 0)$ as an $R$-module, and the proposition follows. 
\end{proof}

\subsection{Calculation of $K(m)^*(T(\star))$}

We now calculate $K(m)^*(T(n))$ using the Atiyah-Hirzebruch spectral sequence by assembling these for all $n$ at once.  The existence of maps $T(i+j) \ra T(i) \sm T(j)$ inducing our multiplication $J(i) \otimes J(j) \ra J(i+j)$ implies that the Atiyah--Hirzebruch spectral sequence
$ \{E_r^{*}(T(\star))\} \Rightarrow K(m)^*(T(\star))$
will be a spectral sequence of differential graded algebras.

By \propref{Qm homology prop} we have the calculation
$$ E_{2^{m+1}}^{*}(T(\star))= \Z/2[v_m^{\pm 1}, x_0,\dots,x_m, x_{m+1}^2, x_{m+2}^2, \dots]/(x_k^{2^{m+1}} \ | \ k \geq 0).$$

\begin{prop} \label{collapse prop}   The spectral sequence collapses after this: $E_{2^{m+1}}^{*}(T(\star)) = E_{\infty}^{*}(T(\star))$.
\end{prop}
\begin{proof}  We show that all the algebra generators of $E_{2^{m+1}}^{*,\star}(T(\star))$ are permanent cycles.  

The generator $x_k \in E_{2^{m+1}}^{1}(T(2^k))$ will be a permanent cycle for $k\leq m$ for degree reasons.

Let $y \in H^2(\C P^{\infty};\Z/2)$ be the nonzero algebra generator. This is certainly a permanent cycle in the A.H.S.S.\ computing $K(m)^*(\C P^{\infty})$. Then $x_k^2 \in E_{2^{m+1}}^{2}(T(2^{k+1}))$ will be also be a permanent cycle, as it will be in the image of $y \in E_{2^{m+1}}^2(\C P^{\infty})$ under any map $f:T(2^{k+1}) \ra \Sinfty \C P^{\infty}$ that is nonzero on mod 2 cohomology in degree $2^{k+1}$. We give more detail of this last point. Such a map exists by the Brown--Gitler property of $T(2^{k+1})$. Since $f^*(y^{2^k}) = x_1^{2^{k+1}}$ and $Sq^{2^k}\cdots Sq^2(y) = y^{2^k}$, we deduce that $ Sq^{2^k}\cdots Sq^2(f^*(y)) = x_1^{2^{k+1}}$.  Thus $f^*(y) = x_k^2$, the only nonzero element in $J(2^{k+1})$ of degree 2.
\end{proof}

\begin{rem}   We confess that what we have done here appears (in dual form) in the literature.  Yamaguchi \cite{Yamaguchi} computes $K(m)_*(\Omega^2 S^{2r+1})$ which, up to a shift of degrees of generators, is equivalent to calculating $K(m)^*(J(\star))$.  The link here is that $\Sinfty \Omega^2 S^{2r+1}$ decomposes as an infinite wedge of finite spectra dual to the $T(n)$ after appropriate suspension. (It is a significant theorem  that these wedge summands really are dual to Brown and Gitler's $T(n)$'s \cite{brown peterson, hunter kuhn}.)  Yamaguchi's argument is similar to ours; indeed our \propref{Qm homology prop} and \propref{collapse prop} are $p=2$ versions of his Lemma (2.1) and Lemma (2.2), though our argument showing that $x_k^2$ is an A.H.S.S.~ permanent cycle is quite different than his: he uses that his spectral sequence is a spectral sequence of Hopf algebra, and we use the Brown-Gitler property of the $T(2^k)$'s.  Ravenel revisits Yamaguchi's argument in \cite{ravenel K(n)}, and both of them conclude that (formulated our way) $K(m)^*(T(\star))$ is isomorphic as a $K(m)^*$--algebra to 
$K(m)^*[\tilde x_0,\dots,\tilde x_m, c_{m+2}, c_{m+3}, \dots]/(\tilde x_k^{2^{m+1}}, c_l^{2^m})$, 
with $\tilde x_k \in K(m)^1(T(2^k))$ and $c_{l} \in K(m)^2(T(2^{l}))$ respectively represented by $x_k$ and $x_{l-1}^2$ in the A.H.S.S.  
\end{rem}

\subsection{Specialization to $m \leq 2$}

We remind the reader that $T(0) = S^0$ and $T(1) = S^1$, and also that $ x(n) \in J(n)^{\alpha_2(n)}$ denotes the class $x_{i_1}\cdots x_{i_d}$ where $n = 2^{i_1} + \cdots + 2^{i_d}$ with $i_1 < \cdots < i_d$ (so $d=\alpha_2(n)$).

It is convenient to let $k_{m,n}$ denote the dimension of $H(J(n);Q_m)$, which also equals the rank of $K(m)^*(T(n))$ as a $K(m)^*$--module. So, for example, $k_{m,0} = k_{m,1} = 1$ for all $m$.  Then let $\displaystyle k_m(t) = \sum_{n=0}^{\infty} k_{m,n}t^n$.

We specialize \propref{Qm homology prop} to the case $m=0$.
\begin{cor} \label{m=0 cor}  $J(n)$ is $Sq^1$--acyclic for all $n \geq 2$, and thus $T(n)$ is rationally acyclic for all $n \geq 2$.
\end{cor}
\begin{proof} $H^*(J(\star);Q_0) = \Z/2[x_0]/(x_0^2)$ and $k_0(t) = 1+ t$.
\end{proof}
We specialize \propref{Qm homology prop} to the case $m=1$.
\begin{cor} \label{m=1 cor}  $k_{1,n}=2$ for all $n \geq 2$.  For $n \geq 1$ and $e$ equal to 0 or 1, cycle representatives for $H(J(2n+e);Q_1)$ are given by $x_0^ex(n)^2$ and $x_0^ex_1x(n-1)^2$.
\end{cor}
\begin{proof}  From the computation
$$ H(J(\star);Q_1) = \bigotimes_{k=0}^1 \Z/2[x_k]/(x_k^4) \otimes \bigotimes_{k=2}^{\infty}\Z/2[x_k^2]/(x_k^4),$$
we compute that
\begin{equation*}
\begin{split}
k_1(t) &  = \frac{(1-t^4)}{(1-t)}\frac{(1-t^8)}{(1-t^2)} \cdot \prod_{k=2}^{\infty} \frac{(1-t^{2^{k+2}})}{(1-t^{2^{k+1}})}\\
  &  =  \frac{(1-t^4)}{(1-t)}\frac{(1-t^8)}{(1-t^2)} \cdot \frac{1}{(1-t^8)}  \\
  &  = \frac{(1+t^2)}{(1-t)} = (1+t^2)\cdot \sum_{n=0}^{\infty} t^n = 1 + t + \sum_{n=2}^{\infty} 2t^n,
\end{split}
\end{equation*}
implying $k_1(n) = 2$ for all $n \geq 2$.  The given elements {\em are} $Q_1$--cycles in $J(n)$ of different degrees.
\end{proof}

We specialize \propref{Qm homology prop} to the case $m=2$.
\begin{cor} \label{m=2 cor} $k_{2,2n+e} = 2n$ for all $n \geq 1$ and $e$ equal to 0 or 1.
\end{cor}
\begin{proof}  We first show that $\displaystyle k_2(t) = \frac{(1+t^4)}{(1-t)(1-t^2)}$ in a manner similar to our last proof.  We have shown that 
$$ H(J(\star);Q_2) = \bigotimes_{k=0}^2\Z/2[x_k]/(x_k^8) \otimes \bigotimes_{k=3}^{\infty} \Z/2[x_k^2]/(x_k^8),$$
and so
\begin{equation*}
\begin{split}
k_2(t)  & = \prod_{k=0}^2 \frac{(1 - t^{2^{k+3}})}{(1-t^{2^k})} \cdot \prod_{k=3}^{\infty} \frac{(1 - t^{2^{k+3}})}{(1-t^{2^{k+1}})}  \\
  & = \frac{(1-t^8)}{(1-t)} \frac{(1-t^{16})}{(1-t^2)} \frac{(1-t^{32})}{(1-t^4)}\cdot \frac{1}{(1-t^{16})} \frac{1}{(1-t^{32})} \\
 & = \frac{(1+t^4)}{(1-t)(1-t^2)}.
\end{split}
\end{equation*}

Now we show that $\displaystyle \frac{(1+t^4)}{(1-t)(1-t^2)} = (1+t)\left(1 + \sum_{n=1}^{\infty} 2nt^{2n}\right)$, as the right hand side has coefficients as in the statement of the corollary.

Recall that $\displaystyle \frac{1}{(1-u)} = \sum_{i=0}^{\infty} u^i$, and then that $\displaystyle \frac{1}{(1-u)^2} = \sum_{i=0}^{\infty} (i+1)u^i$.  Multiplying by $2u$ and letting $n=i+1$, we see that
$\displaystyle \frac{2u}{(1-u)^2} = \sum_{n=1}^{\infty} 2nu^n$.  Letting $u=t^2$, we learn that $\displaystyle \frac{2t^2}{(1-t^2)^2} = \sum_{n=1}^{\infty} 2nt^{2n}$.  Altogether, this implies
\begin{equation*}
\begin{split}
(1+t)\left(1 + \sum_{n=1}^{\infty} 2nt^{2n}\right) & = (1+t)\left(1 + \frac{2t^2}{(1-t^2)^2}\right)  \\
  &  = (1+t)\frac{1+t^4}{(1-t^2)^2} = \frac{(1+t^4)}{(1-t)(1-t^2)}.
\end{split}
\end{equation*}
\end{proof}

\corref{m=2 cor} has an obvious consequence.

\begin{cor} \label{X(n,m) not K2 acyclic cor} If $m$ and $n$ are even and $m > n$, then $f(n,m): T(m) \ra T(n)$ is not a $K(2)^*$--isomorphism and thus $X(n,m)$ is not $K(2)^*$-acyclic.
\end{cor}

\section{Proof of \thmref{main thm}: part 1} \label{main proof part 1}

In this and the next two sections, we determine for which pairs $(n,m)$, the map
$$ q(n,m)^*: J(n) \otimes (t) \ra J(m)$$
induces an isomorphism 
$$ q(n,m)^*: H(J(n) \otimes (t); Q_1) \ra H(J(m);Q_1).$$
Equivalently, we determine when the $A$--module $Q(n,m)$ is $Q_1$--acyclic.

First of all, $q(n, m)^*$ is not a $Q_1$-isomorphism if $0 \leq n \leq 1$, since then $H(J(n) \otimes (t); Q_1)$ is one-dimensional and concentrated in degree 2 or 3, and this is not isomorphic to $H(J(m); Q_1)$ for any $m$. Thus, we assume hereafter that $n \geq 2$.

From \cite{kuhn dl}, we learn that the short exact sequence 
$$ 0 \ra J(m) \ra Q(n,m) \ra \Sigma^{-1}J(n) \ra 0$$
splits when $m<2n-1$, so that $Q(n,m) \simeq J(m) \oplus \Sigma^{-1}J(n)$ and is not $Q_1$--acyclic.  Furthermore, when $m=2n-1$, the sequence identifies with one desuspension of a Mahowald sequence, and we learn that $Q(n,2n-1) \simeq \Sigma^{-1}J(2n)$ and so also not $Q_1$--acyclic.  Summarizing:
\begin{prop}  If $q(n,m)^*$ is a $Q_1$--isomorphism, then $m \geq 2n$.
\end{prop}

Now comes our first use of a key trick.  

Since $H(F(1);Q_1) = \langle t^2, t^8, t^{16}, \dots \rangle$ while $H((t);Q_1) = \langle t^2 \rangle$, the inclusion $F(1) \hra (t)$ clearly induces a surjection on $Q_1$--homology.  The next lemma follows.
  
\begin{lem} Let $\bar q(n,m): J(n) \otimes F(1) \ra J(m)$ be the composite $$J(n) \otimes F(1) \hra J(n) \otimes (t) \xra{q(n,m)^*} J(m).$$  Then $q(n,m)^*$ will be a $Q_1$--isomorphism if and only if $\bar q(n,m)$  induces a surjection on $Q_1$--homology.
\end{lem}

Under the condition that $m \geq 2n$, the map $\bar q(n,m)$ turns out to be much easier to describe than $q(n,m)^*$.

\begin{lem} \label{q bar lemma}  Let $m \geq 2n$.  If we write $m$ as $m=l+2^k$ with $l<2^k$, then $\bar q(n,m)$ agrees with the composite
\begin{equation} \label{q bar factorization} J(n) \otimes F(1) \xra{p(n,l)^* \otimes p_k} J(l) \otimes J(2^k) \xra{\mu} J(m).
\end{equation}
\end{lem}
\begin{proof}  $\bar q(n,m)$ is the unique $A$--module map that is nonzero on every basis element of $(J(n) \otimes F(1))^m$, so we need to show that our composite has this same property.  Elements of the form $x \otimes t^{2^k}$, with $x$ a standard basis element of $J(n)^{l}$, certainly map nonzero:
$$ \bar q(n,m)(x\otimes t^{2^k}) = \mu(p(n,l)^*(x) \otimes p_k(t^{2^k})) = \mu(x_0^{l} \otimes x_0^{2^k}) = x_0^m.$$
We now observe that our condition that $m\geq 2n$ implies that every basis element of $[J(n) \otimes F(1)]^m$ has this form: if $x \otimes t^{2^i} \in [J(n) \otimes F(1)]^m$ then $i=k$.  If $i>k$, then $|x| = l + (2^k - 2^i) \leq l - 2^k < 0$, which can't happen.  If $i<k$, then  
$2n \geq 2|x| = 2(l+2^k-2^i) \geq 2(l+2^{k-1}) = 2l + 2^k > l+2^k = m$,
which also can't happen.
\end{proof}

\begin{prop}  \label{reduce from q to p prop} Let $m \geq 2n$.  If we write $m$ as $m=l+2^k$ with $l<2^k$, then $q(n,m)^*$ will be a $Q_1$--isomorphism if and only if $p(n,l)^*$ is a $Q_1$--isomorphism.
\end{prop}
\begin{proof}  We need to check that the composite (\ref{q bar factorization}) is a $Q_1$--surjection if and only if $p(n,l)^*$ is a $Q_1$--isomorphism.  

The first observation is that the image of $p(n,l)^* \otimes p_k$ in $Q_1$-homology will be precisely $(\IM p(n,l)^*) \otimes \langle x_{k-1}^2 \rangle$. If $0 \leq l \leq 1$, then neither $p(n,l)^*$ nor (\ref{q bar factorization}) is a $Q_1$-surjection since $\dim(H(J(l); Q_1)) = 1$. If $l \geq 2$, then $\mu: J(l) \otimes J(2^k) \ra J(m)$ induces an isomorphism $\mu_*: H(J(l);Q_1) \otimes \langle x_{k-1}^2\rangle \simeq H(J(m);Q_1)$, and the proposition follows.
\end{proof}

\section{Proof of \thmref{main thm}: part 2} \label{main proof part 2}

Now we focus on the case when our pair are both even, so we are determining for which pairs $(n,l)$ the map
$$ p(2n,2l)^*: H(J(2n); Q_1) \ra H(J(2l);Q_1)
$$
is an isomorphism of graded vector spaces of total dimension 2.

For starters, one clearly needs that these two graded vector spaces are isomorphic.  Call $(n,l)$ a {\em good pair} if this is the case.

\begin{prop} \label{good pair prop} $(n,l)$ is a good pair if and only if $\alpha_2(l) = \alpha_2(n)$ and $\nu_2(l) = \nu_2(n)$.
\end{prop}
\begin{proof}  $H(J(2l); Q_1)$ is represented by the classes 
$x(l)^2$ and $ x_1x(l-1)^2$ in degrees 
$2\alpha_2(l)$ and  $2\alpha_2(l-1) +1$, and, similarly,
$H(J(2n); Q_1)$ is represented by the classes $x(n)^2$ and $x_1x(n-1)^2$ in degrees $2\alpha_2(n)$ and $2\alpha_2(n-1) +1$.

Matching degrees and simplifying, we find that $(n,l)$ will be good if and only if 
$$ \alpha_2(l) = \alpha_2(n) \text{\hspace{.1in} and \hspace{.1in}} \alpha_2(l-1)  = \alpha_2(n-1),
$$
or, equivalently, if
$$ \alpha_2(l) = \alpha_2(n) \text{\hspace{.1in} and \hspace{.1in}} \alpha_2(l-1) - \alpha_2(l)  = \alpha_2(n-1) - \alpha_2(n).
$$
The next observation finishes the proof.
\end{proof}
\begin{lem} \label{alpha nu lemma}
For all $n\geq 1$, $\alpha_2(n-1) - \alpha_2(n) = \nu_2(n)-1$.
\end{lem}

Now suppose that $\alpha_2(l) = \alpha_2(n)$ and $\nu_2(l) = \nu_2(n)$.   If we let $d = \alpha_2(n)-1$ and $i = \nu_2(n)$, then these two conditions say that $n$ and $l$ have the form
$$ l = 2^i + 2^{j_1} + \cdots + 2^{j_d} \text{ with } i < j_1 < \cdots < j_d$$
and 
$$ n = 2^i + 2^{k_1} + \cdots + 2^{k_d} \text{ with } i < k_1 < \cdots < k_d.$$

\begin{thm} \label{p(2n,2l) iso thm}  In this situation, 
$ p(2n,2l)^*: H(J(2n);Q_1) \ra H(J(2l);Q_1)$
is an isomorphism if 
\begin{equation*}\tag{$\clubsuit$} j_1 \leq k_1 < j_2 \leq k_2 < \cdots < j_d \leq k_d
\end{equation*}
and is zero otherwise.
\end{thm}

The next lemma says that the combinatorial condition of this theorem can be packaged, for our purposes, in a rather simple way.  To explain how we discovered this, see Remark \ref{magic lemma remark}.

\begin{lem} \label{magic lemma}  Suppose $\alpha_2(m) = \alpha_2(n)+1$, $m=l+2^k$ with $l<2^k$, and write
$$ l = 2^{j_0} + \cdots + 2^{j_d} \text{ with } j_0 < \cdots < j_d$$
and 
$$ n = 2^{k_0} + \cdots + 2^{k_d} \text{ with } k_0 < \cdots < k_d.$$
Then
$\binom{n-m-1}{n} = 1 \mod 2$ if and only if $j_0 \leq k_0 < j_1 \leq k_1 < \cdots < j_d \leq k_d$ and $m\geq 2n$.
\end{lem}

\begin{cor} Suppose $m \geq 2n$, and let $m=l+2^k$ with $l<2^k$. Then $ p(2n,2l)^*: H(J(2n);Q_1) \ra H(J(2l);Q_1)$
is an isomorphism if an only if $\alpha_2(m) = \alpha_2(n) + 1$, $\nu_2(m) = \nu_2(n)$, and $\binom{m-n-1}{n} = 1 \mod 2$.
\end{cor} 

Note that this corollary combines with \propref{reduce from q to p prop} to prove \thmref{main thm}(a).

We postpone the proof of \lemref{magic lemma} to the end of the section.

\thmref{p(2n,2l) iso thm} follows immediately from the next two propositions.  To state these, we need some definitions and notation. 

Let $J = \{j_1, \dots, j_d\}$ and $K = \{k_1, \dots k_d\}$.  Let $S_d$ denote the $d$th symmetric group, and then let  $S(J,K) \subset S_d$ be the set
$$ S(J,K) = \{ \sigma \in S_d \ | \ k_c \geq j_{\sigma(c)}  \text{ for all } c \},$$
with cardinality $|S(J,K)|$.

\begin{prop} \label{prop 1} $ p(2n,2l)^*: H(J(2n);Q_1) \ra H(J(2l);Q_1)$ is `multiplication by $|S(J,K)|$'. More precisely, 
$$ p(2n,2l)^*(x(n)^2) = |S(J,K)| x(l)^2 \text{ and }$$ 
$$p(2n,2l)^*(x_1x(n-1)^2) = |S(J,K)| x_1x(l-1)^2.$$
\end{prop}

\begin{prop} \label{prop 2}  $S(J,K)$ has the single element `identity' if ($\clubsuit$) holds, and has an even number of elements otherwise. 
\end{prop}

\begin{proof}[Proof of \propref{prop 1}]  For notational simplicity, in this proof we let $p=p(2n,2l)^*: J(2n) \ra J(2l)$.  Our proof makes heavy use of \lemref{pk lemma} which said that $p_k(t^{2^j}) = x_{k-j}^{2^j}$ for $j \leq k$, where $p_k: F(1) \ra J(2^k)$ is the nonzero $A$--module map.  

Let $g: J(2^{i+1}) \otimes F(1)^{\otimes d} \ra J(2n)$ be the composite
$$ J(2^{i+1}) \otimes F(1)^{\otimes d} \ra J(2^{i+1}) \otimes J(2^{k_1+1}) \otimes \cdots \otimes J(2^{k_d+1}) \xra{\mu} J(2n),$$

where the first map is $\id \otimes p_{k_1+1} \otimes \cdots \otimes p_{k_d+1}$ and the second is multiplication.

The map $g$ induces an epimorphism in $Q_1$--homology; more precisely,
$$ g(x_i^2 t_1^2 t_2^2 \cdots  t_d^2) = x_i^2x_{k_1}^2 \cdots x_{k_d}^2= x(n)^2 \text{ \hspace{.1in} and }$$
$$ g(x_1x_0^2\cdots x_{i-1}^2 t_1^2 t_2^2 \cdots  t_d^2) = x_1x_0^2\cdots x_{i-1}^2x_{k_1}^2 \cdots x_{k_d}^2= x_1x(n-1)^2.$$
Here $t_c$ is the one-dimensional class in the $c$th copy of $F(1)$ in the $d$--fold tensor product.

Therefore, it suffices to show that
\begin{equation} \label{pg 1} p(g(x_i^2 t_1^2 t_2^2 \cdots  t_d^2)) = |S(J,K)| x(l)^2 
\end{equation}
and
\begin{equation} \label{pg 2} p(g(x_1x_0^2\cdots x_{i-1}^2 t_1^2 t_2^2 \cdots  t_d^2)) = |S(J,K)| x_1x(l-1)^2.
\end{equation}

To do this, we identify $p \circ g: J(2^{i+1}) \otimes F(1)^{\otimes d} \ra J(2l)$ in a way that makes this easy to see.

Recall that $p\circ g$ is determined by its values on $(J(2^{i+1}) \otimes F(1)^{\otimes d})^{2l}$.  Note that this has basis $ \mathcal B = \{ b_{\sigma} \ | \ \sigma \in S_d \}$
where $b_{\sigma} = x_0^{2^{i+1}}t_1^{2^{1+j_{\sigma(1)}}}\cdots t_d^{2^{1+j_{\sigma(d)}}}$.

For all $c$,
$$ p_{k_c+1}(t^{2^{1+j_{\sigma(c)}}}) =
\begin{cases}
x_{k_c - j_{\sigma(c)}}^{2^{1+j_{\sigma(c)}}} & \text{if } k_c \geq j_{\sigma(c)} \\ 0 & \text{otherwise},
\end{cases}
$$ 
and it follows that $g(b_{\sigma})$ is a standard basis element in $J(2n)^{2l}$ for $\sigma \in S(J,K)$ and 0 otherwise. 
This implies
\begin{equation} \label{pg equation}
p(g(b_{\sigma})) =
\begin{cases}
x_0^{2l} & \text{if } \sigma \in S(J,K) \\ 0 & \text{otherwise}.
\end{cases}
\end{equation}

Now we identify the basis of $\Hom_{A}(J(2^{i+1}) \otimes F(1)^{\otimes d}, J(2l))$ dual to $\mathcal B$. 

Similar to our definition of $g$, we let $h: J(2^{i+1}) \otimes F(1)^{\otimes d} \ra J(2l)$  be the composite
$$ J(2^{i+1}) \otimes F(1)^{\otimes d} \ra J(2^{i+1}) \otimes J(2^{j_1 +1}) \otimes \cdots \otimes J(2^{j_d+1}) \xra{\mu} J(2l).$$
Similar to, but simpler than, our computation of $g(b_{\sigma})$, one computes that 
$$
h(b_{\sigma}) =
\begin{cases}
x_0^{2l} & \text{if } \sigma = \id \\ 0 & \text{otherwise}.
\end{cases}
$$

It follows that if we let $h_{\tau}$ be the composite
$$ J(2^{i+1}) \otimes F(1)^{\otimes d} \xra{\id \otimes \tau^{-1}} J(2^{i+1}) \otimes F(1)^{\otimes d} \xra{h} J(2l),$$
for $\tau \in S_d$, then
\begin{equation} \label{htau equation}
h_{\tau}(b_{\sigma}) =
\begin{cases}
x_0^{2l} & \text{if } \sigma = \tau \\ 0 & \text{otherwise}.
\end{cases}
\end{equation}

Comparing (\ref{pg equation}) with (\ref{htau equation}), we conclude that $\displaystyle p\circ g = \sum_{\sigma \in S(J,K)} h_{\sigma}$.

Now one checks that for all $\sigma \in S_d$, one has 
\begin{equation} \label{h 1} h_{\sigma}(x_i^2 t_1^2 t_2^2 \cdots  t_d^2) = x(l)^2 
\end{equation}
and
\begin{equation} \label{h 2} h_{\sigma}(x_1x_0^2\cdots x_{i-1}^2 t_1^2 t_2^2 \cdots  t_d^2) = x_1x(l-1)^2.
\end{equation}
Comparing these to (\ref{pg 1}) and (\ref{pg 2}), the proof of \propref{prop 1} is complete.

\end{proof}

\begin{proof}[Proof of \propref{prop 2}]
Recall that $j_1 < \cdots < j_d$, $k_1< \cdots < k_d$, and we wish to show that the set
$$ S(J,K) = \{\sigma \in S_d \ | \ j_{\sigma(c)} \leq k_c\}$$
will consist of only the identity element if 
\begin{equation*}\tag{$\clubsuit$} j_1 \leq k_1 < j_2 \leq k_2 < \cdots < j_d \leq k_d,
\end{equation*} and will have an even number of elements otherwise.

For $1 \leq c \leq d$, define a function $s$ by $s(c) = |\{ b \ | \ j_b \leq k_c\}|$. Then $0 \leq s(1) \leq \cdots \leq s(d) \leq d$, and if $s(c) > 0$ then $s(c)$ is the maximal $b$ such that $j_b \leq k_c$.   

From this it follows that  $S(J,K) = S(s)$ where 
$$S(s) = \{ \sigma \in S_d \ | \ \sigma(c) \leq s(c) \text{ for all } c \}.$$

If $s(c)<c$ for any $c$, e.g.\ if $s(c)=0$, $S(s)$ will be empty.  The condition $s(c) \geq c$ for all $c$ corresponds to the condition $j_c \leq k_c$ for all $c$.

Here is an interpretation of $S(s)$: it is set of the ways of fitting blocks of size $1, 2, \dots, d$ under the graph of $s$.

If any two of the $s(c)$'s equal each other, then one clearly gets an even number of ways of fitting the blocks in. (More generally our set is acted on freely by an easily defined symmetry group which will have even order if any two values of the $s(c)$'s are equal.)

If all the $s(c)$'s are positive and distinct then $s(c)=c$ for all $c$, and this holds if and only if ($\clubsuit$) holds. In this case, one easily sees that the identity element is the only element in $S(s)$. 
\end{proof}

\begin{proof}[Proof of \lemref{magic lemma}]  Recall that $m = l + 2^k$ with $l<2^k$, and $n$ and $l$ have binary decompositions  
$$l = 2^{j_0} + \cdots + 2^{j_d}, j_0 < \cdots < j_d,\text{ and } n = 2^{k_0} + \cdots + 2^{k_d}, k_0 < \cdots < k_d.$$
Our goal is to show that $\binom{m-n-1}{n} = 1 \mod 2$ if and only if $m \geq 2n$ and also
\begin{equation}\tag{$\diamondsuit$} j_0 \leq k_0 < j_1 \leq k_1 < \cdots < j_d \leq k_d. 
\end{equation}

Note that $\binom{m-n-1}{n} = 1 \mod 2$ implies that $m-n-1 \geq n$, and thus $m\geq 2n$.  We also note that $m \geq 2n$ holds if and only if $k>k_d$: if $k>k_d$, then $m \geq 2^k \geq 2^{k_d} \geq 2n$, and if $m > 2n$, then $2^{k+1} > l+2^k = m \geq 2n \geq 2^{k_d+1}$.

It therefore suffices to show that $\binom{m-n-1}{n} = 1 \mod 2$ if and only if ($\diamondsuit$), assuming that $k>k_d$.

Now $m-n-1 = (2^k-1-n) + l$, and $2^k-1-n$ is the sum of all powers of $2$ up to $2^{k-1}$ {\em except} those of the form $2^{k_c}$ for $c=0, \dots, d$.  We need to show that if we add $l$ to this, the resulting sum contains $2^{k_c}$ for all $c$ in its binary expansion if and only if ($\diamondsuit$).

It is convenient to let $k_{-1} = -1$.  For a fixed $c$, if no $j_b$ satisfies $k_{c-1} < j_b \leq k_c$, then $(2^k-1-n+l)$ will not contain $2^{k_c}$ in its binary expansion, and $\binom{m-n-1}{n} = 0 \mod 2$.  By the pigeonhole principle, the other alternative is precisely condition ($\diamondsuit$).  In this case one has
\begin{equation*}
\begin{split}
2^k-1-n+l &
= \sum_{c=0}^d  [(\sum_{k_{c-1} < b < k_c} 2^b) + 2^{j_c}]+ \sum_{k_d<b\leq k-1} 2^b \\
  & =
\sum_{c=0}^d  [(\sum_{k_{c-1} < b < j_c} 2^b) + 2^{k_c}]+ \sum_{k_d<b\leq k-1} 2^b,
\end{split}
\end{equation*}
and so $2^{k_c}$ is in the binary expansion of $2^k-1-n+l$ for all $c$, and thus $\binom{m-n-1}{n} = 1 \mod 2$.
\end{proof}

\section{Proof of \thmref{main thm}: part 3} \label{main proof part 3}

\thmref{main thm}(c) asserted that $Q(n,m)$ can only be a $Q_1$-acyclic if $m$ and $n$ are of the same parity. Thanks to \propref{reduce from q to p prop}, to prove this it suffices to prove the next two propositions.

\begin{prop} \label{even odd prop}  $p(2n,2l+1)^*: J(2n) \ra J(2l+1)$ does not induce an isomorphism in $Q_1$--homology, since $p(2n,2l+1)^*(x(n)^2) = 0$.
\end{prop}

\begin{prop} \label{odd even prop}  $p(2n+1,2l)^*: J(2n+1) \ra J(2l)$ does not induce an isomorphism in $Q_1$--homology, since $p(2n+1,2l)^*(x_0 x(n)^2) = 0$.
\end{prop}

\begin{proof}[Proof of \propref{even odd prop}]
For notational simplicity, in this proof we let $p=p(2n, 2l + 1)^*: J(2n) \ra J(2l + 1)$. As we are assuming $2n \geq 2$, if $l = 0$ then $p$ is certainly not a $Q_1$-isomorphism because the dimensions of the source and target differ. Thus, we may assume $l\geq 1$ for the remainder of the proof.

As in Section~\ref{main proof part 2}, we begin by identifying the \emph{good pairs}, i.e.\ the pairs $(n, l)$ for which the graded vector spaces $H(J(2n); Q_1)$ and $H(J(2l + 1); Q_1)$ are isomorphic. The former is represented by the classes $x(n)^2$ and $x_1 x(n - 1)^2$ in degrees $2\alpha_2(n)$ and $2\alpha_2(n - 1) + 1$, and the latter by the classes $x_0 x(l)^2$ and $x_0 x_1 x(l - 1)^2$ in degrees $2\alpha_2(l) + 1$ and $2\alpha_2(l - 1) + 2$.

Matching degrees and simplifying, we find that $(n, l)$ is a good pair if and only if
$$
\alpha_2(n) = \alpha_2(l - 1) + 1 \text{\hspace{.1in} and \hspace{.1in}} \alpha_2(n - 1) = \alpha_2(l),
$$
or, making use of \lemref{alpha nu lemma}, if and only if
$$
\alpha_2(n) = \alpha_2(l - 1) + 1 \text{\hspace{.1in} and \hspace{.1in}} \nu_2(n) + \nu_2(l) = 1.
$$

We consider these two cases separately. \\

\noindent \textbf{Case \#1}: $\nu_2(n) = 0$ and $\nu_2(l) = 1$.

These conditions, together with the condition $\alpha_2(n) = \alpha_2(l - 1) + 1$, imply that $n$ and $l$ have the form
$$
l = 2 + 2^{j_1} + \cdots + 2^{j_d} \text{ with } 1 < j_1 < \cdots < j_d
$$
and
$$
n = 1 + 2^{k_1} + \cdots + 2^{k_{d+1}} \text{ with } 0 < k_1 < \cdots < k_{d+1}
$$
where $d := \alpha_2(l) - 1 = \alpha_2(n) - 2$.

Let $g : J(2) \otimes F(1)^{\otimes d+1} \to J(2n)$ be the composite
$$
J(2) \otimes J(2^{k_1 + 1}) \otimes \cdots J(2^{k_{d+1} + 1}) \xrightarrow{\mu} J(2n),
$$
where the first map is $\id \otimes p_{k_1 + 1} \otimes \cdots \otimes p_{k_{d+1} + 1}$ and the second is multiplication. Since
$$
g(x_0^2 \otimes t_1^2 \otimes t_2^2 \otimes \cdots \otimes t_{d+1}^2) = x_0^2 x_{k_1}^2 \cdots x_{k_{d+1}}^2 = x(n)^2,
$$
(where $t_c$ denotes the one-dimensional class in the $c$th copy of $F(1)$ in the $(d+1)$-fold tensor product), it suffices to show that
$$
p(g(x_0^2 \otimes t_1^2 \otimes \cdots \otimes t_{d+1}^2)) = 0.
$$

To do this, we will produce a basis for $\Hom_A(J(2) \otimes F(1)^{\otimes d+1}, J(2l + 1))$ and show that each element of this basis maps $x_0^2 \otimes t_1^2 \otimes \cdots \otimes t_{d+1}^2$ to zero. A map of this form is determined by its values on $(J(2) \otimes F(1)^{\otimes d+1})^{2l + 1}$, which has basis $\mathcal{B} = \{b_{\sigma} \,|\, \sigma \in S_{\{0, \ldots, d\}}\}$, where
$$
b_{\sigma} = x_1 \otimes t^{2^{j_{\sigma(0)} + 1}} \otimes t^{2^{j_{\sigma(1)} + 1}} \otimes \cdots \otimes t^{2^{j_{\sigma(d)} + 1}}
$$
and $j_0 := 1$. Now we identify the basis of $\Hom_{A}(J(2) \otimes F(1)^{\otimes d+1}, J(2l + 1))$ dual to $\mathcal{B}$. Let $h : J(2) \otimes F(1)^{\otimes d+1} \to J(2l + 1)$ be the composite
$$
J(2) \otimes F(1)^{\otimes d+1} \to J(1) \otimes J(4) \otimes J(2^{j_1 + 1}) \otimes \cdots \otimes J(2^{j_d + 1}) \xrightarrow{\mu} J(2l + 1),
$$
where the first map is $p(2, 1)^* \otimes p_2 \otimes p_{j_1 + 2} \otimes \cdots \otimes p_{j_d + 2}$, and then let $h_{\tau}$ be the composite
$$
J(2) \otimes F(1)^{\otimes d+1} \xrightarrow{\id \otimes \tau^{-1}} J(2) \otimes F(1)^{\otimes d+1} \xrightarrow{h} J(2l + 1)
$$
for $\tau \in S_{d+1}$. One checks that
$$
h_{\tau}(b_{\sigma}) = \begin{cases}
    x_0^{2l+1} & \text{if }\sigma = \tau, \\
    0 & \text{otherwise,} \\
\end{cases}
$$
verifying that $\{h_{\tau} \,|\, \tau \in S_{d+1}\}$ is the basis dual to $\mathcal{B}$. But then
$$
h_{\tau}(x_0^2 \otimes t_1^2 \otimes \cdots \otimes t_{d+1}^2) = h(x_0^2 \otimes t_1^2 \otimes \cdots \otimes t_{d+1}^2) = 0
$$
for all $\tau \in S_{d+1}$, for the simple reason that $p(2, 1)^*(x_0^2) = 0$. From this it follows that any map $J(2) \otimes F(1)^{\otimes d+1} \to J(2l + 1)$, in particular the map $p \circ g$, annihilates $x_0^2 \otimes t_1^2 \otimes \cdots \otimes t_{d+1}^2$. \\

\noindent \textbf{Case \#2}: $\nu_2(n) = 1$ and $\nu_2(l) = 0$.

In this case, $n$ and $l$ must have the form
$$
l = 1 + 2^{j_1} + \cdots + 2^{j_d} \text{ with } 0 < j_1 < \cdots < j_d
$$
and
$$
n = 2 + 2^{k_1} + \cdots + 2^{k_d} \text{ with } 1 < k_1 < \cdots < k_d
$$
where $d := \alpha_2(l) - 1 = \alpha_2(n) - 1$. Let $g$ be the composite
$$
J(4) \otimes F(1)^{\otimes d} \to J(4) \otimes J(2^{k_1 + 1}) \otimes \cdots \otimes J(2^{k_d + 1}) \xrightarrow{\mu} J(2n),
$$
where the first map is $\id \otimes p_{k_1 + 1} \otimes \cdots \otimes p_{k_d + 1}$. Since
$$
g(x_1^2 \otimes t_1^2 \otimes \cdots \otimes t_d^2) = x_1^2 x_{k_1}^2 \cdots x_{k_d}^2 = x(n)^2,
$$
it suffices to show that $p(g(x_1^2 \otimes t_1^2 \otimes \cdots \otimes t_d^2)) = 0$.

A basis for $(J(4) \otimes F(1)^{\otimes d})^{2l + 1}$ is given by $\{b_{\sigma}\}_{\sigma \in S_d}$, where
$$
b_{\sigma} = x_0^2 x_1 \otimes t^{2^{k_{\sigma(1)} + 1}} \otimes \cdots \otimes t^{2^{k_{\sigma(d)} + 1}},
$$
and the corresponding basis of $\Hom_A(J(4) \otimes F(1)^{\otimes d}, J(2l + 1))$ is $\{h_{\tau}\}_{\tau \in S_d}$, where $h_{\tau}$ is the map
$$
J(4) \otimes F(1)^{\otimes d} \xrightarrow{1 \otimes \tau^{-1}} J(4) \otimes F(1)^{\otimes d} \xrightarrow{h} J(2l + 1)
$$
and $h = \mu \circ (p(4, 3)^* \otimes p_{j_1 + 1} \otimes \cdots \otimes p_{j_d + 1})$. But then
$$
h_{\tau}(x_1^2 \otimes t_1^2 \otimes \cdots \otimes t_d^2) = h(x_1^2 \otimes t_1^2 \otimes \cdots \otimes t_d^2) = 0
$$
because
$$
p(4, 3)^*(x_1^2) = p(4, 3)^*(Sq^1(x_2)) = Sq^1(p(4, 3)^*(x_2)) = Sq^1(0) = 0.
$$
It follows that any map $J(4) \otimes F(1)^d \to J(2l + 1)$, in particular $p \circ g$, annihilates $x_1^2 \otimes t_1^2 \otimes \cdots \otimes t_d^2$.
\end{proof}

\begin{proof}[Proof of \propref{odd even prop}]
As with \propref{even odd prop}, we can assume $l \geq 1$. By reversing the roles of $n$ and $l$ in the proof of \propref{even odd prop}, we find that $(n, l)$ is a \textit{good pair}, i.e.\ $H(J(2n + 1); Q_1)$ and $H(J(2l); Q_1)$ are isomorphic, if and only if
$$
\alpha_2(n) = \alpha_2(l - 1) \text{\hspace{.1in} and \hspace{.1in}} \nu_2(n) + \nu_2(l) = 1.
$$
So we consider two cases. \\

\noindent \textbf{Case \#1}: $\nu_2(n) = 0$ and $\nu_2(l) = 1$.

In this case, $n$ and $l$ have the form
$$
l = 2 + 2^{j_1} + \cdots + 2^{j_d} \text{ with } 1 < j_1 < \cdots < j_d
$$
and
$$
n = 1 + 2^{k_1} + \cdots + 2^{k_d} \text{ with } 0 < k_1 < \cdots < k_d,
$$
where $d = \alpha(n) - 1 = \alpha(l) - 1$. Let $g$ be the composite
$$
J(3) \otimes F(1)^{\otimes d} \to J(3) \otimes J(2^{k_1 + 1}) \otimes \cdots \otimes J(2^{k_d + 1}) \xrightarrow{\mu} J(2n + 1),
$$
where the first map is $\id \otimes p_{k_1 + 1} \otimes \cdots \otimes p_{k_d + 1}$. Since
$$
g(x_0^3 \otimes t_1^2 \otimes \cdots \otimes t_d^2) = x_0^3 x_{k_1}^2 \cdots x_{k_d}^2 = x_0 x(n)^2,
$$
it suffices to show that $p(2n + 1, 2l)^*(g(x_0^3 \otimes t_1^2 \otimes \cdots \otimes t_d^2)) = 0$.

Now $p(2n + 1, 2l)^* \circ g$ is a map $J(3) \otimes F(1)^{\otimes d} \to J(2l)$, but the only such map is the zero map since $(J(3) \otimes F(1)^{\otimes d})^{2l} = 0$; to see this, observe that $\alpha(2l - i) = d + 1$ for $2 \leq i \leq 3$ (the range of degrees in which $J(3)$ is nontrivial). \\

\noindent \textbf{Case \#2}: $\nu_2(n) = 1$ and $\nu_2(l) = 0$.

In this case, $n$ and $l$ have the form
$$
l = 1 + 2^{j_1} + \cdots + 2^{j_{d+1}} \text{ with } 0 < j_1 < \cdots < j_{d+1}
$$
and
$$
n = 2 + 2^{k_1} + \cdots + 2^{k_d} \text{ with } 1 < k_1 < \cdots < k_d,
$$
where $d = \alpha(n) - 1 = \alpha(l) - 2$. Let $g$ be the composite
$$
J(5) \otimes F(1)^{\otimes d} \to J(5) \otimes J(2^{k_1 + 1}) \otimes \cdots \otimes J(2^{k_d + 1}) \xrightarrow{\mu} J(2n + 1),
$$
where the first map is $\id \otimes p_{k_1 + 1} \otimes \cdots \otimes p_{k_d + 1}$. Since
$$
g(x_0 x_1^2 \otimes t_1^2 \otimes \cdots \otimes t_d^2) = x_0 x_1^2 x_{k_1}^2 \cdots x_{k_d}^2 = x_0 x(n)^2,
$$
it suffices to show that $p(2n + 1, 2l)^*(g(x_0 x_1^2 \otimes t_1^2 \otimes \cdots \otimes t_d^2)) = 0$.

Now $p(2n + 1, 2l)^* \circ g$ is a map $J(5) \otimes F(1)^{\otimes d} \to J(2l)$, but the only such map is the zero map since $(J(5) \otimes F(1)^{\otimes d})^{2l} = 0$; to see this, observe that $\alpha(2l - i) \geq d + 1$ for $2 \leq i \leq 5$ (the range of degrees in which $J(5)$ is nontrivial).
\end{proof}

\section{Applications} \label{applications}

\subsection{Proof of \thmref{K theory thm}}

We prove the various statements in \thmref{K theory thm}.  
\begin{proof}[Proof of \thmref{K theory thm}(a)]
Let $n$ be even,  $\nu_2(n) = i$ and suppose $ n = 2^i + 2^{k_1} + \cdots + 2^{k_d}$ with  $i < k_1 < \cdots < k_d$, so $d = \alpha_2(n)-1$.  It is convenient to let $n_c = 2^i + \sum_{j=1}^c 2^{k_j}$ for $c = 0,1, \dots d$, so $2^i = n_0 < \cdots < n_d = n$. 

If we let $f:T(n) \ra T(2^i)$ be the composite of the maps
$$ T(n) \xra{f(n_{d-1}, n_d)} T(n_{d-1}) \xra{f(n_{d-2}, n_{d-1})} \cdots \xra{f(n_1,n_2)} T(n_1) \xra{f(n_0,n_1)} T(n_0),$$
then $f$ is the composite of $L_1$--equivalences, so is an $L_1$--equivalence.  As the composition of $d$ maps, each inducing 0 in mod 2 cohomology, $f$ has Adams filtration at least $d$.  
\end{proof}
\begin{proof}[Proof of \thmref{K theory thm}(b)]
As discussed at the end of \secref{HZ/2 section}, there is a Mahowald cofibration sequence
$$ T(2^{i-1}) \xra{\alpha} T(2^i) \xra{\beta^{\prime}} \Sigma T(2i-1).$$
Now $\Sigma T(2i-1) \simeq \Sigma^2 T(2^i -2)$, and by part (a), there is an $L_1$--equivalence $f: T(2^i-2) \ra T(2)$. Therefore, if we let $\beta$ be the composite $\Sigma^2 f\circ \beta^{\prime}$, then the sequence
\begin{equation} \label{L1 Mahowald cofib} T(2^{i-1}) \xra{\alpha} T(2^i) \xra{\beta} \Sigma^2 T(2)
\end{equation}
will be a cofibration sequence after $L_1$--localization.
\end{proof}
\begin{proof}[Proof of \thmref{K theory thm}(c)]  
For $i \geq 1$, we prove that $K^0(T(2^i))_{(2)} \simeq \Z/2^i$, and $K^1(T(2^i))_{(2)} \simeq 0$ by induction on $i$. Recalling that $T(2) = \Sinfty \R P^2$, one sees that this is true for $i=1$.

For the inductive step, assume the calculation for $i-1$.  Because the 2--local $K$--theory of both $T(2^{i-1})$ and $T(2)$ is concentrated in degree 0, the sequence (\ref{L1 Mahowald cofib})
will induce a short exact sequence
\begin{equation} \label{K-theory ses}  0 \ra K^0(T(2))_{(2)} \ra K^0(T(2^i))_{(2)} \ra K^0(T(2^{i-1}))_{(2)} \ra 0,
\end{equation}
and we also learn that $K^1(T(2^i))_{(2)} \simeq 0$.

In particular, we learn that $K^0(T(2^i))_{(2)}$ is a finite abelian group of order $2^i$.  This group must be cyclic, since $K(1)^*(T(2^i))$ can be calculated from this group by the universal coefficient theorem, and we know independently that $K(1)^*(T(2^i))$ is 2--dimensional.  

It follows that $K^0(T(2^i))_{(2)} \simeq \Z/{2^i}$ and (\ref{K-theory ses}) must be equivalent to
\[0 \ra \Z/2 \ra \Z/2^i \ra \Z/2^{i-1} \ra 0.\qedhere\]
\end{proof}

\subsection{Connections to the thesis of Brian Thomas}

Mod 2 homology applied to the Goodwillie tower of the functor $X \rightsquigarrow \Sinfty \Oinfty X$ yields a spectral sequence converging to $H_*(\Oinfty X)$ for all 0--connected spectra $X$.  The Hopf algebra $H_*(\Oinfty X)$ is a module over the Dyer--Lashof algebra $\mathcal R$, and, as studied in \cite{kuhn mccarty}, this is reflected in the spectral sequence: $E^{\infty}_{*,*}(X)$ is always a primitively generated Hopf algebra equipped with Dyer--Lashof operations compatible with those acting on $H_*(\Oinfty X)$..  

Let $\displaystyle E(X) = \bigoplus_{s=0}^{\infty} \Ext_A^{s,s}(H^*(X), J(\star))$.   \cite[Cor. 1.14]{kuhn mccarty} implies that, if a certain `geometric condition' holds and $E(X)$ is concentrated in even degrees and is generated by $\Hom_A(H^*(X),J(\star))$ as an $\mathcal R$--module, then 
$$ E^{\infty}_{*,*}(X) = S^*(E(X))/(Q_0x-x^2),$$ 
where we let $Q_0x = Q^{|x|}x$.

In Thomas' thesis \cite{Brian Thomas thesis}, his goal is to verify that this applies when $X = \Sigma^2 ku$, so that $\Oinfty X = BU$.  One knows that $H^*(ku) = A \otimes_{E(1)} \Z/2$, and Thomas carefully computes $\displaystyle E(\Sigma^2 ku) = \bigoplus_{s=0}^{\infty} \Ext_{E(1)}^{*,*}(\Sigma^2 \Z/2, J(\star))$ with some of its $\mathcal R$--module structure. 

The punchline is that $E(\Sigma^2 ku)$ has a basis of elements 
$$\bar b_n(k) \in \Ext_A^{\alpha_2(n)-1+k, \alpha_2(n)-1+k}(H^*(\Sigma^2 ku), J(2^{k+1}n))$$
with $n \geq 1$ and $k \ge 0$,  and one has the following formula: if $ n = 2^i + 2^{j_1} + \cdots + 2^{j_d}$ with  $i < j_1 < \cdots < j_d$, then 
$$\bar b_n(k) = Q_0^kQ^{2^{j_d}}\cdots Q^{2^{j_1}}\bar b_{2^i}(0).$$
(The behavior of $Q_0$ here is a bit more than is proved in \cite{Brian Thomas thesis}, and uses \cite[Theorem 1.7]{kuhn dl} applied to calculation in \cite{Brian Thomas thesis}.)

One concludes that $E^{\infty}_{*,*}(\Sigma^2ku) = \Z/2[\bar b_1(0), \bar b_2(0), \dots]$. This recovers the known calculation of $H_*(BU)$, usually computed using very different means, along with some of its Dyer-Lashof operations.

A key to the calculations is to use \thmref{2nd main thm}(b) to prove algebraic analogues of statements (a) and (b) of \thmref{K theory thm}.

\begin{prop}  If $ n = 2^i + 2^{j_1} + \cdots + 2^{j_d}$ with  $1 \leq i < j_1 < \cdots < j_d$, then 
$$Q^{2^{j_d}}\cdots Q^{2^{j_1}}: \Ext_{E(1)}^{s,s}(\Sigma^2 \Z/2, J(2^i)) \ra \Ext_{E(1)}^{s-1+\alpha(n), s-1+\alpha(n)}(\Sigma^2 \Z/2, J(n)).$$
is an isomorphism.
\end{prop}

\begin{prop}  The Mahowald sequences induce short exact sequences
\begin{equation*} 
\begin{split} 
0 \ra \Ext_{E(1)}^{s,s}(\Z/2,J(2)) & \ra \Ext_{E(1)}^{s+i,s+i}(\Sigma^2\Z/2, J(2^{i+1})) \\
& \ra \Ext_{E(1)}^{s+i,s+i}(\Sigma^2\Z/2, J(2^{i})) \ra 0.
\end{split}
\end{equation*}
\end{prop}

More details about these calculations will appear elsewhere.

\begin{rem} \label{magic lemma remark}  It is known \cite{kochman}, that the Dyer-Lashof operations acting on $H_*(BU) = \Z/2[b_1, b_2, \dots]$, with $b_n$ of degree $2n$, satisfy:
$$ Q^{2r}b_n = \binom{r-1}{n} b_{n+r} \mod \text{decomposables},$$
and this, together with \thmref{p(2n,2l) iso thm}, suggested \lemref{magic lemma} (by letting $m=n+r$).
\end{rem}


\begin{thebibliography}{KMcC13}

\bibitem[BE20]{prasit and egger} P.~Bhattacharya and P.~Egger, {\em
A class of 2-local finite spectra which admit a $v_2^1$-self-map}, Adv. Math. {\bf 360} (2020), 106895, 40 pp.

\bibitem[BG73]{bg} E.~H.~Brown and S.~Gitler,{\em  A Spectrum whose
Cohomology is a Certain Cyclic Module over the Steenrod Algebra}, Topology {\bf 12}
(1973), 283--295.

\bibitem[BG78]{brown peterson} E.~H.~Brown and F.P.~Peterson, {\em 
On the stable decomposition of $\Omega^2S^{r+2}$}, Trans.A.M.S.~ {\bf 243} (1978), 287–298.

\bibitem[CMM78]{cmm78}  F.~R.~Cohen, M.~Mahowald, and R.~J.~Milgram, {\em
The Stable Decomposition of the Double Loop Space of a Sphere},
A.M.S.~ Proc. Symp. Pure Math. {\bf 32} (1978), 225--228.
%

\bibitem[G85]{goerss T(n) is spacelike} P.~G.~Goerss,  {\em A Direct construction for the
duals of Brown-Gitler spectra}, Ind. U.
Math. J. {\bf 34} (1985), 733--751.

\bibitem[GLM93]{glm} P.~Goerss, J.~Lannes, and F.~Morel, {\em Hopf Algebras, Witt
Vectors, and Brown-Gitler Spectra}, A.M.S. Cont. Math.
{\bf 146} (1993), 111--128.

\bibitem[HS98]{hs} M.~J.~ Hopkins and J.~H.~Smith, {\em Nilpotence and stable homotopy theory. II}, Ann. Math. {\bf 148} (1998), 1--49.
    
\bibitem[K73]{kochman} S.~O.~Kochman, {\em Homology of the classical groups over the Dyer--Lashof algebra}, Trans.A.M.S.~ {\bf 185} (1973), 83--136.

\bibitem[HK00]{hunter kuhn} D.~J.~Hunter and N.~J.~Kuhn, {\em Characterizations of spectra with U-injective cohomology which satisfy the Brown--Gitler property}, Trans. A.M.S. {\bf 352} (2000), 1171--1190.
    
\bibitem[K23]{kuhn dl} N.~J.~Kuhn, {\em Dyer-Lashof operations as extensions of Brown Gitler modules}, preprint, 2023. arXiv:2306.14158. To appear in Homology Homotopy Appl. 

\bibitem[KL24a]{kuhn lloyd} N.~J.~Kuhn and C.~J.~R,~Lloyd, {\em Chromatic fixed point theory and the Balmer spectrum for extraspecial 2--groups}, Amer. J. Math. {\bf 146} (2024), 769--812.  
    
\bibitem[KL24b]{kuhn lloyd AGT} N. ~J. ~Kuhn and C.~J.~R.~Lloyd, {\em Computing the Morava K-theory of real Grassmanians using chromatic fixed point theory}, Alg. Geo. Top. {\bf 24} (2024, 919--950.

\bibitem[KMcC13]{kuhn mccarty} N.~J.~Kuhn and J.~B.~McCarty, {\em The mod 2 homology of infinite loopspaces}, Algebraic and Geometric Topology {\bf 13} (2013), 687--745.

\bibitem[LZ87]{lz} J.~Lannes and S.~Zarati, {\em Sur les foncteurs d\'eriv\'es de la d\'estabilisation}, Math. Zeit. {\bf 194} (1987), 25--59. (Correction: {\bf 10} (2010).)

\bibitem[L88]{lannes T(n) are spacelike} J.~Lannes, {\em Sur le $n$--dual du $n$--\`eme spectre de Brown--Gitler}, Math.Zeit. {\bf 199}(1988), 29--42.

\bibitem[M58]{milnor}  J.~Milnor, {\em  The Steenord algebra and its dual},
Ann. Math. {\bf 67} (1958), 150--171.


\bibitem[M85]{mitchell} S. ~A. ~Mitchell, {\em Finite complexes with $A(n)$--free cohomology}, Topology {\bf 24}(1985), 227--248.
    
\bibitem[P96]{palmieri} J.~H.~Palmieri, {\em Nilpotence for modules over the mod 2 Steenrod algebra, II}, Duke Math.J. {\bf 82}, 209--226.

\bibitem[Rav92]{ravenel orange book}D. ~C. ~Ravenel, {\em Nilpotence and periodicity in stable homotopy theory}, Annals of Mathematics Studies 128, Princeton University Press, 1992.

\bibitem[Rav93]{ravenel K(n)} D. ~C. ~Ravenel, {\em The homology and Morava K-theory of $\Omega^2 SU(n)$}, Forum Math. {\bf 1}(1993), 1–21.

\bibitem[S94]{schwartz book} L.~Schwartz, {\em Unstable modules over the Steenrod algebra and Sullivan's fixed point conjecture}, Chicago Lecture Series in Math, University of Chicago Press, 1994.

\bibitem[T19]{Brian Thomas thesis} B.~Thomas, {\em Dyer--Lashof operations as extensions and an application to $H_*(BU)$}, PhD thesis, University of Virginia, 2019. Available at \url{https://libraetd.lib.virginia.edu/downloads/8g84mm76s?filename=Thomas_Brian_Dissertation.pdf}

\bibitem[Y80]{Yagita} N.~Yagita, {\em On the Steenrod algebra of Morava $K$-theory}, J. London Math. Soc. (2) {\bf 22} (1980), no. 3, 423--438.

\bibitem[Y88]{Yamaguchi} A.~Yamaguchi, {\em Morava $K$-theory of double loopspaces of spheres}, Math. Zeit.{\bf 199} (1988), 511--523.





\end{thebibliography}
\end{document}